\documentclass{article}

\usepackage{arxiv}

\usepackage[utf8]{inputenc} 
\usepackage[T1]{fontenc}    
\usepackage{hyperref}       
\usepackage{url}            
\usepackage{booktabs}       
\usepackage{amsfonts}       
\usepackage{nicefrac}       
\usepackage{microtype}      
\usepackage{lipsum}		
\usepackage{graphicx}
\usepackage{mathtools}
\usepackage{amssymb}
\usepackage{amsthm}
\usepackage{amsmath}
\usepackage{doi}

\usepackage{cases}
\usepackage{microtype}
\usepackage[utf8]{inputenx}
\usepackage{amssymb}
\usepackage{amsthm}
\usepackage{mathtools, cases}
\usepackage{mathbbol}
\usepackage{makecell}
\usepackage{dsfont}
\usepackage{bbm}
\usepackage{mathabx}
\allowdisplaybreaks

\renewcommand{\sinh}{\operatorname{sinh}}
\renewcommand{\cosh}{\operatorname{cosh}}

\newtheorem{thm}{Theorem}

\title{One-dimensional and planar random motions with variable propagation speeds}

\author{ 
	\href{https://orcid.org/0000-0002-6163-044X}{Enzo Orsingher} \\
	Department of Statistical Sciences\\
	Sapienza University of Rome\\
	\texttt{enzo.orsingher@uniroma1.it}\\
	\And
\href{0000-0002-6421-533X}{Manfred Marvin Marchione}\\
	Independent Researcher\\
	\texttt{manfredmarvin.marchione@gmail.com}}

\date{\today}

\renewcommand{\headeright}{}
\renewcommand{\undertitle}{}

\begin{document}
\maketitle

\begin{abstract}
In this paper, we study univariate and planar random motions with variable propagation speeds. We first consider motions with space-varying velocity, which can be reduced to constant-velocity motions by means of suitable nonlinear transformations. We examine a special case of a motion which is confined within the unit interval. To provide a general expression of the moments of this process, we introduce a new family of polynomials which generalize the classical Euler polynomials. We then examine a planar extension of this process which moves along orthogonal directions. A process with velocity depending on the direction is also examined, and its mean conditional on the initial direction and the number of direction changes is given in terms of confluent hypergeometric functions. We conjecture that, in the hydrodynamic limit, this process is absorbed at a point in an arbitrarily small time. We finally study a motion with time-dependent velocity. We prove that this process can be represented as an integral with respect to a standard telegraph process, and we obtain its covariance function explicitly. Moreover, we show that this process behaves as a It\^{o} integral with respect to Brownian motion in the hydrodynamic limit.
\end{abstract}


\section{Introduction}
\noindent While random motions with constant velocity have been extensively studied in the literature, random motions with variable velocity in space and time have received little attention. The case of velocity depending on space and time was first considered by Masoliver and Weiss \cite{masoliverweiss}. The authors discussed the construction of a persistent random walk on a one-dimensional lattice, and they proved that, by taking the limit in the continuum under a suitable scaling, the distribution of the random walk is asymptotically described by a telegraph-type equation with non-constant velocity. The case of space-varying velocity was further explored by Garra and Orsingher \cite{garraorsingher}, who studied different variants of the process, such as that with drift or that in which the direction changes are governed by a non-homogeneous Poisson process. A common approach to studying the distribution of such processes is to reduce them to constant-velocity motions through suitable nonlinear transformations.\\
\noindent In this paper, we study one-dimensional and planar random motions with non-constant velocity. In particular, we study both the cases in which the velocity varies with space and with time, as well as a special case in which the functional form of the velocity depends on the direction along which the process is moving. Our analysis starts by reviewing and extending the study of random motions with space-varying velocity. Hence, for a given continuously differentiable positive-valued function $\mathtt{v}:\mathbb{R}\to\mathbb{R}^+$, we consider a one-dimensional stochastic process $\{X(t)\}_{t\ge0}$ which can move rightwards with velocity $\mathtt{v}(X(t))$ or leftwards with velocity $-\mathtt{v}\left(X(t)\right)$. At the initial time $t=0$, the process lies at a point $x_0$ and starts moving leftwards or rightwards with probability $\frac{1}{2}$. The process then changes direction at Poisson times. The Poisson process which governs the direction changes is assumed to have constant intensity $\lambda>0$. It can be verified that the paths of the process $X(t)$ satisfy the following ordinary differential equation almost everywhere \begin{equation*}\frac{d X(t)}{dt}=\pm \mathtt{v}(X(t))\end{equation*} where the sign of the right-hand side of the equation depends on whether the process is moving leftwards or rightwards. We emphasize that the paths are not differentiable at the instants when direction changes occur. However, since these times are determined by a Poisson process, they form a set of Lebesgue measure zero, and thus the paths are almost everywhere differentiable. Our first result is that the probability density function $f(x,t)$ of the process $X(t)$ is a solution to the partial differential equation
\begin{equation}\label{intro:pde}\frac{\partial^2 f}{\partial t^2}+2\lambda \frac{\partial f}{\partial t}=\frac{\partial}{\partial x} \left\{\mathtt{v}(x)\frac{\partial}{\partial x}\big[\mathtt{v}(x)\,f\big]\right\},\qquad x\in\mathbb{R},\;t>0.\end{equation}
Observe that equation (\ref{intro:pde}) slightly differs from that previously obtained by other authors like Masoliver and Weiss \cite{masoliverweiss} and Garra and Orsingher \cite{garraorsingher}, in the sense that the velocity and the spatial differential operator appear interchanged. The ambiguity arises because the process $X(t)$ can be expressed as a nonlinear transformation of a standard telegraph process with constant velocity. The formulations proposed by the aforementioned authors do not account for the differential term resulting from this change of variables. The correct equations consistent with (\ref{intro:pde}) were first derived by Tailleur and Cates \cite{tailleur} by means of physical arguments, and were later studied by Angelani et al. \cite{angelani2} and Angelani and Garra \cite{angelani1}. In this paper, a detailed probabilistic derivation of equation (\ref{intro:pde}) is presented.\\
\noindent It is well-known that, denoting by $\{\mathcal{T}(t)\}_{t\ge0}$ a standard telegraph process with constant velocity $c>0$ and initial value $X(0)=0$, the following representation holds
\begin{equation}\mathcal{T}(t)=c\int_{x_0}^{X(t)}\frac{1}{\mathtt{v}(w)}\,dw.\label{intro:TXrelat}\end{equation}
Clearly, since the distribution of $\mathcal{T}(t)$ is widely known, formula (\ref{intro:TXrelat}) permits us to obtain the distribution of $X(t)$ explicitly. In this paper, we are interested in studying finite-velocity random motions which are confined in bounded domains. A similar problem was investigated by Angelani and Garra \cite{angelani2}. The approach adopted by these authors consists in solving equation (\ref{intro:pde}) with suitable boundary conditions, in order to impose barriers to the process. In contrast, we show that choosing suitable velocity functions naturally yields processes whose support lies within a bounded region. We discuss a special case of space-varying velocity, namely the case $\mathtt{v}(x)=c\,x(1-x)$. Clearly, this velocity function vanishes linearly at both points $x=0$ and $x=1$. As we will discuss later, this implies that these points are unreachable in finite time. Hence, under the assumption that $x_0\in(0,1)$, the process is naturally confined within the interval $(0,1)$, even for arbitrarily large values of $t$. Indeed, the following representation holds for $X(t)$:
 \begin{equation}\label{intro:01proc}X(t)=\frac{x_0\,e^{\mathcal{T}(t)}}{1-x_0\,\left(1-e^{\mathcal{T}(t)}\right)}.\end{equation}
Hence, the exact distribution of $X(t)$ can be obtained easily. However, due to the nonlinearity of the relationship between $X(t)$ and $\mathcal{T}(t)$, finding the moments of $X(t)$ is a more challenging task. Since we are interested in finding an explicit expression of the moments of arbitrary order, we introduce in the present paper a family of polynomials $E_n^{(a,\theta)}(x)$, with $a,\theta>0$, with generating function
\begin{equation}\left(\frac{\theta+1}{\theta e^t+1}\right)^ae^{xt}=\sum_{n=0}^{+\infty}\frac{E_n^{(a,\theta)}(x)}{n!}\,t^n.\label{intro:taylorseriesstatement}\end{equation} 

\noindent These polynomials provide a generalization of the well-known Euler polynomials $E_n(x),\, n\in\mathbb{N}$, (see Gradshteyn and Ryzhik \cite{gradshteyn}). We emphasize that an extension of the Euler polynomials involving the parameter $a$ has already been studied in the literature (see Roman \cite{roman}). In this paper, we further generalize these polynomials by introducing the additional parameter $\theta$. We obtain the explicit general expression of our generalized polynomials and we show that, similarly to the classical Euler polynomials, they form an Appell sequence. Moreover, by writing that
\begin{equation*}X(t)^a=x_0^a\sum_{n=0}^{+\infty}\frac{E_n^{\left(a,\frac{x_0}{1-x_0}\right)}(a)}{n!}\;\mathcal{T}(t)^n\end{equation*} we are able to express the moments of $X(t)$ in terms of infinite series. In particular, we show that
\begin{align}\mathbb{E}\left[X(t)^a\right]=x_0^a\;\sqrt{\frac{\lambda t}{2}}\,e^{-\lambda t}\sum_{n=0}^{+\infty}&\frac{\Gamma\left(n+\frac{1}{2}\right)\;E_{2n}^{\left(a,\frac{x_0}{1-x_0}\right)}(a)}{(2n)!}\cdot\nonumber\\
&\qquad \cdot\left[I_{n+\frac{1}{2}}(\lambda t)+I_{n-\frac{1}{2}}(\lambda t)\right]\left(\frac{2c^2t}{\lambda}\right)^n.\label{intro:moma}\end{align}
We also study the hydrodynamic limit of the process $X(t)$ for $\lambda,c\to+\infty$, with $\frac{\lambda}{c^2}\to1$, which reads
$$\lim_{\lambda,c\to+\infty}X(t)\overset{i.d.}{=}\frac{x_0\,e^{B(t)}}{1-x_0\,\left(1-e^{B(t)}\right)}$$  where $\{B(t)\}_{t\ge0}$ is a standard Brownian motion. Thus, by taking the hydrodynamic limit of formula (\ref{intro:moma}), we prove that
$$\mathbb{E}\left[\left(\frac{x_0\,e^{B(t)}}{1-x_0\,\left(1-e^{B(t)}\right)}\right)^a\right]=\frac{x_0^a}{\sqrt{\pi}}\sum_{n=0}^{+\infty}\frac{\Gamma\left(n+\frac{1}{2}\right)\;E_{2n}^{\left(a,\frac{x_0}{1-x_0}\right)}(a)}{(2n)!}\;\left(2t\right)^n.$$

Our next step is to investigate a planar extension of random motions with space-varying velocity. In contrast to Garra and Orsingher \cite{garraorsingher}, who discussed a planar motion with an infinite number of directions, we study the case of a planar random motion with orthogonal directions. Thus, we consider a bivariate process $(X(t),\,Y(t))$ which can move along four possible directions $d_j,\;j=0,1,2,3$, namely
$$d_j=\left(\cos\left(\frac{\pi j}{2}\right),\;\sin\left(\frac{\pi j}{2}\right)\right)$$ where $d_j=d_{j+4n}$ for all integers $n$. At the initial time $t=0$, the process lies at the point $(x_0,y_0)$ and starts moving along one of the four possible directions, chosen randomly with equal probability $\frac{1}{4}$. The direction changes are paced by a homogeneous Poisson process $N(t)$ with intensity $\lambda$. Each time a Poisson event occurs, the process $(X(t),\,Y(t))$ changes direction either clockwise, from $d_j$ to $d_{j-1}$, or counterclockwise, from $d_j$ to $d_{j+1}$, each with probability $\frac{1}{2}$. The key feature of our model is that the velocity of the motion depends on the current position along the direction of motion. For a fixed continuously differentiable function $\mathtt{v}(\cdot)$, the velocity is $\mathtt{v}(X(t))$ when the process is moving horizontally, while it is $\mathtt{v}(Y(t))$ when the process is moving vertically. Under this assumption, the process $(X(t),\,Y(t))$ can be mapped into a random motion with constant velocity $c>0$ by means of the transformation 
\begin{equation}\begin{cases}U(t)=c\,\int_{x_0}^{X(t)}\frac{1}{\mathtt{v}(z)}\,dz\\[1.2ex] V(t)=c\,\int_{y_0}^{Y(t)}\frac{1}{\mathtt{v}(z)}\,dz.\end{cases}\label{introUVXY}\end{equation} The stochastic process $(U(t),\,V(t))$ is therefore a random motion with constant velocity and orthogonal directions, whose distribution has been studied by several authors (see, for example, Orsingher and Marchione \cite{orsinghermarchione}). Therefore, using formula (\ref{introUVXY}) permits us to obtain the distribution of $(X(t),\,Y(t))$ in explicit form. In particular, denoting by $R_t$ the support of the distribution of $(X(t),\,Y(t))$ and by $\text{Int}(R_t)$ its interior, it holds that 
\begin{align}&f(x,y,t)=\frac{e^{-\lambda t}}{2\,\mathtt{v}(x)\,\mathtt{v}(y)}\cdot\Bigg[\frac{\lambda}{2}\, I_0\Bigg(\frac{\lambda}{2}\sqrt{t^2-\left(\int_{x_0}^x\frac{dz}{\mathtt{v}(z)}+\int_{y_0}^y\frac{dz}{\mathtt{v}(z)}\right)^2}\;\Bigg)\nonumber\\
&\qquad\qquad+\frac{\partial}{\partial t}I_0\Bigg(\frac{\lambda}{2}\sqrt{t^2-\left(\int_{x_0}^x\frac{dz}{\mathtt{v}(z)}+\int_{y_0}^y\frac{dz}{\mathtt{v}(z)}\right)^2}\;\Bigg)\Bigg]\nonumber\\
&\qquad\qquad\quad\cdot\Bigg[\frac{\lambda}{2}\, I_0\Bigg(\frac{\lambda}{2}\sqrt{t^2-\left(\int_{x_0}^x\frac{dz}{\mathtt{v}(z)}-\int_{y_0}^y\frac{dz}{\mathtt{v}(z)}\right)^2}\Bigg)\nonumber\\
&\qquad\qquad\qquad+\frac{\partial}{\partial t}I_0\Bigg(\frac{\lambda}{2}\sqrt{t^2-\left(\int_{x_0}^x\frac{dz}{\mathtt{v}(z)}-\int_{y_0}^y\frac{dz}{\mathtt{v}(z)}\right)^2}\Bigg)\Bigg],\; (x,y)\in \text{Int}(R_t).\nonumber\end{align}
We emphasize that the distribution of the process has a singular component on the boundary of its support. As we will see, to describe the distribution of the process on the boundary $\partial R_t$, it is convenient to define a suitable coordinate system. We provide a few alternatives. However, we believe that the most intuitive way to identify the position of $(X(t),\,Y(t))$ on $\partial R_t$ is to use its abscissa $X(t)$. Restricting the analysis to the first quadrant of the cartesian plane for simplicity, we prove that \begin{align}\mathbb{P}\big(\,(X&(t),\,Y(t))\in\partial R_t,\;X(t)>0,\;Y(t)>0,\;X(t)\in dx\big)/dx\nonumber\\=&\frac{e^{-\lambda t}}{2\,\mathtt{v}(x)}\Bigg[\frac{\lambda}{2}\,I_0\left(\lambda\;\sqrt{\int_{x_0}^x\frac{dz}{\mathtt{v}(z)}\,\left(t-\int_{x_0}^x\frac{dz}{\mathtt{v}(z)}\right)}\;\right)\nonumber\\&\qquad\qquad+\frac{\partial}{\partial t}\,I_0\left(\lambda\;\sqrt{\int_{x_0}^x\frac{dz}{\mathtt{v}(z)}\,\left(t-\int_{x_0}^x\frac{dz}{\mathtt{v}(z)}\right)}\;\right)\Bigg].\nonumber\end{align} We then examine the special case
$$\mathtt{v}(x)=c\,(1-x^2)$$ where formula (\ref{introUVXY}) reduces to
$$\begin{cases}X(t)=\tanh\left(U(t)\right)\\Y(t)=\tanh\left(V(t)\right).\end{cases}$$ In this case, we give an explicit representation of the support $R_t$, and we provide the exact distributions in the interior and the boundary of $R_t$. We note that, similarly to the univariate case discussed above, the bivariate process $(X(t),\,Y(t))$ is now confined within the square $[-1,1]^2$. This follows from the fact that the selected velocity function vanishes with linear rate at the points $x=\pm1$.\\
The next problem we investigate in this paper concerns a random motion in which the velocity varies with the current direction. In particular, we consider a finite-velocity univariate random motion $X(t)$ and we denote by $D(t)$ the direction of $X(t)$. We say that $D(t)=d_0$ when $X(t)$ is moving rightwards, while $D(t)=d_1$ when $X(t)$ is moving leftwards. Given two positive-valued continuously differentiable functions $\mathtt{v}_0(\cdot)$ and $\mathtt{v}_1(\cdot)$, we assume that the process moves with velocity $\mathtt{v}_0\left(X(t)\right)$ when it moves rightwards, that is when $D(t)=d_0$, and with velocity $-\mathtt{v}_1\left(X(t)\right)$ when $D(t)=d_1$. As usual, we assume that the direction changes are governed by a homogenous Poisson process $N(t)$ with intensity $\lambda$. 
Denoting by $f_j(x,t)$, for $j=0,1$, the density functions $$f_j(x,t)=\mathbb{P}\left(X(t)\in\mathop{dx},\,D(t)=d_j\right)/\mathop{dx}$$ it can be verified that the following system of partial differential equations holds
\begin{equation}
\begin{dcases}
\frac{\partial f_0}{\partial t}=-\frac{\partial}{\partial x}\big[\mathtt{v}_0(x)\,f_0\big]+\lambda f_1-\lambda f_0\\
\frac{\partial f_1}{\partial t}=\frac{\partial}{\partial x}\big[\mathtt{v}_1(x)\,f_1\big]+\lambda f_0-\lambda f_1.
\end{dcases}\label{intro:directiondependentsystem}
\end{equation}
Unfortunately, in contrast to the usual case where the functional form of the velocity does not depend on the direction, combining the equations of the system (\ref{intro:directiondependentsystem}) to obtain a second-order partial differential equation for the probability density function $f(x,t)=f_0(x,t)+f_1(x,t)$ seems not possible. Moreover, finding the exact distribution of $X(t)$ in general is a difficult task. To illustrate some characteristics of the process, we simplify the analysis by considering a special case. In particular, we consider the case
$$\mathtt{v}_0(x)=c\,(1-x),\qquad \mathtt{v}_1(x)=c\,x$$ with $c>0$, and we assume that $X(0)=x_0$, with $x_0\in(0,1)$. denoting by $T_k,\;k\ge0$, the arrival times of the Poisson process $N(t)$, that is $$T_k=\inf\{t:\;N(t)\ge k\}$$ we show that the following recursive relationship holds the $X(t)$
$$X(T_k)=X(T_{k-1})\,e^{-c(T_k-T_{k-1})}+\left(1-e^{-c(T_k-T_{k-1})}\right)\,\mathbbm{1}\{D(T_{k-1})=d_0\},\qquad k\ge1.$$ By solving the recursion above, we give an explicit representation of the position of the process $X(t)$ in terms of the intertimes of the Poisson process $N(t)$. This representation permits us to obtain the mean of $X(t)$ conditional on the initial direction and the number of direction changes. In particular, we prove that, for $n\in\mathbb{N}$,
\begin{flalign*}
\mathbb{E}\left[X(t)\big\lvert D(0)=d_0,\, N(t)=2k\right]=1+x_0 e^{-ct}-e^{-ct}\sum_{j=0}^{2k}(-1)^j{}_1F_1\big(j;2k+1;ct\big)
\end{flalign*}
\vspace{-12mm}

\begin{flalign*}
\mathbb{E}\left[X(t)\big\lvert D(0)=d_0,\, N(t)=2k+1\right]=x_0e^{-ct}-e^{-ct}\sum_{j=0}^{2k+1}(-1)^j{}_1F_1\big(j;2(k+1);ct\big)
\end{flalign*}
\vspace{-12mm}

\begin{flalign*}
\mathbb{E}\left[X(t)\big\lvert D(0)=d_1,\, N(t)=2k\right]=-(1-x_0)e^{-ct}+e^{-ct}\sum_{j=0}^{2k}(-1)^j{}_1F_1\big(j;2k+1;ct\big)
\end{flalign*}
\vspace{-12mm}

\begin{flalign*}
\mathbb{E}&\left[X(t)\big\lvert D(0)=d_1,\, N(t)=2k+1\right]\\&\qquad\qquad=1-(1-x_0)e^{-ct}+e^{-ct}\sum_{j=0}^{2k+1}(-1)^j{}_1F_1\big(j;2(k+1);ct\big).\qquad\qquad\qquad\end{flalign*}

where ${}_1F_1(a;b;z)$ is the confluent hypergeometric function
$${}_1F_1(a;b;z)=\frac{\Gamma(b)}{\Gamma(a)\,\Gamma(b-a)}\,\int_0^1e^{zu}u^{a-1}(1-u)^{b-a-1}du,\;\; z\in\mathbb{C}$$ where $\Re(a)>0,\;\Re(b)>0$. It follows immediately that the unconditional mean reads
\begin{equation}\label{intro:unconditionaln}\mathbb{E}\left[X(t)\right]=x_0e^{-ct}+\frac{1-e^{-ct}}{2}.\end{equation}
Formula (\ref{intro:unconditionaln}) provides some interesting information concerning the behaviour of the process in the hydrodynamic limit where $\lambda,c\to+\infty$ with the usual scaling $\frac{\lambda}{c^2}\to1$. Indeed, observe that the limiting mean is $$\lim_{\lambda,c\to+\infty}\mathbb{E}\left[X(t)\right]=\frac{1}{2}$$ which is notably independent from the starting point $x_0$. As we illustrate in various examples, in the standard case where the velocity function is independent of the direction, finite-velocity random motions typically converge to diffusion processes in the hydrodynamic limit. However, in our model, the fact that the limiting mean is independent of the starting point $x_0$ is not compatible with the behaviour of a nontrivial diffusion process, for which the distribution of the process is concentrated near the starting point and spreads around the initial value over time. Hence, when the functional form of the velocity of the process depends on the current direction, the hydrodynamic limit is not guaranteed to be a diffusion process in general. Moreover, observe that the asymptotic mean is independent of $t$. Therefore, even if the limiting process starts at a point $x_0\neq\frac{1}{2}$, its mean approaches $\frac{1}{2}$ in an arbitrarily small time, suggesting some degenerate behaviour of $X(t)$ in the hydrodynamic limit. A rigorous investigation of this phenomenon is left for future research.\\
In the final part of the paper, we investigate a univariate random motion with time-dependent velocity. In particular, we study a stochastic process $\{X(t)\}_{t\ge0}$ moving with velocity $$\mathtt{v}(t)=c\,\sigma(t)$$ where $c$ is a stricly positive constant and $\sigma:[0,+\infty)\to\mathbb{R}^+$ is a positive-valued function of time. As usual, the process is assumed to change direction at Poisson times with intensity $\lambda$. The probability density function of $X(t)$ is a solution to the partial differential equation
\begin{equation}\frac{\partial}{\partial t}\left[\frac{1}{\sigma(t)}\,\frac{\partial f}{\partial t}\right]+\frac{2\lambda}{\sigma(t)}\,\frac{\partial f}{\partial t}=c^2\sigma(t)\,\frac{\partial^2 f}{\partial x^2}.\label{intro:sigmapde}\end{equation} As discussed by Masoliver and Weiss \cite{masoliverweiss} and Garra and Orsingher \cite{garraorsingher}, finding an explicit solution to equation (\ref{intro:sigmapde}) is challenging. However we show that the following representation holds for $X(t)$
\begin{equation}X(t)=\int_0^t\sigma(u)\,d\mathcal{T}(u)\label{intro:intdT}\end{equation} where $\{\mathcal{T}(t)\}_{t\ge0}$ is a standard telegraph process. In view of the representation (\ref{intro:intdT}), we are able to find the covariance structure of $X(t)$, which reads
\begin{equation}\label{intro:autocovbeforelimit}\mathbb{E}\left[X(s)\,X(t)\right]=c^2\int_0^t\int_0^s e^{-2\lambda\lvert x-y\lvert}\sigma(x)\sigma(y)\mathop{dx}\mathop{dy}.\end{equation}
We study the hydrodynamic limit of $X(t)$ for $\lambda,c\to+\infty,\;\frac{\lambda}{c^2}\to1$, and we show that
\begin{equation*}\lim_{\lambda,c\to+\infty}\mathbb{P}\big(X(t)\in dx\big)/dx=\frac{1}{\sqrt{2\pi \int_0^t\sigma^2(u)\mathop{du}}}\;\exp\left(-\frac{x^2}{2\int_0^t\sigma^2(u)\mathop{du}}\right)\end{equation*} which implies that
\begin{equation}\lim_{\lambda,c\to+\infty} X(t)\overset{i.d.}{=}\int_0^t\sigma(u)\,dB(u).\label{intro:itolimit}\end{equation} Observe that formula (\ref{intro:itolimit}) is consistent with both (\ref{intro:intdT}) and (\ref{intro:autocovbeforelimit}). Indeed, we prove that, under suitable regularity assumptions on $\sigma(\cdot)$, the hydrodynamic limit of formula (\ref{intro:autocovbeforelimit}) coincides with the covariance function of the It\^{o} integral (\ref{intro:itolimit}).

\section{Random motions with space-varying velocity}\label{prelimsec}
\noindent Consider a stochastic process $\{X(t)\}_{t\ge0}$ whose velocity at time $t$ depends on the current position $X(t)$. In particular, for a given positive-valued function $\mathtt{v}:\mathbb{R}\to\mathbb{R}^+$, we assume that the process $X(t)$ can move rightwards with velocity $\mathtt{v}(X(t))$ or leftwards with velocity $-\mathtt{v}\left(X(t)\right)$. At the initial time $t=0$, the process lies at a fixed point $x_0$ and starts moving leftwards or rightwards with probability $\frac{1}{2}$. The process then switches direction at Poisson times, changing the sign of its velocity. Denote by $N(t)$ the Poisson process which governs the direction changes and by $\lambda$ its intensity. Moreover, we denote by $D(t)$ the direction of $X(t)$. In particular, we say that $D(t)=d_0$ when $X(t)$ is moving rightwards, while $D(t)=d_1$ when $X(t)$ is moving leftwards.\\
Throughout the paper, we assume that the velocity function $\mathtt{v}(\cdot)$ is continuously differentiable. Under this assumption, it is clear that the paths of the process $X(t)$ are differentiable almost everywhere. Indeed, when the process is moving rightwards, the position $X(t)$ of the process satisfies, by construction, the ordinary differential equation $\frac{dX(t)}{dt}=\mathtt{v}(X(t))$. Similarly, if $X(t)$ is moving leftwards it must satisfy the equation $\frac{dX(t)}{dt}=-\mathtt{v}(X(t))$. Note that the solutions to these equations are unique under the assumption that $\mathtt{v}(\cdot)$ is continuously differentiable (we refer to the book by Perko \cite{perko} for details). Therefore, the only points at which the paths of $X(t)$ are non-differentiable are the instants at which a change of direction occurs. However, this set has Lebesgue measure zero since it is countable. For this reason, in the remainder of the paper, we will make use of the derivative $\frac{dX(t)}{dt}$.\\
\noindent We also emphasize that the support of the process $X(t)$ is time-varying, and we denote it by $R_t$. The explicit representation of the set $R_t$ depends on the choice of the velocity $\mathtt{v}(\cdot)$.\\
\noindent In order to study the distribution of $X(t)$, we define, for $j=0,1$, the cumulative distribution functions
\begin{equation*}F_j(x,t)=\mathbb{P}\big(X(t)\le x,\,D(t)=d_j\big),\qquad x\in R_t,\,t>0.\end{equation*} Observe that the distribution function of $X(t)$
\begin{equation*}F(x,t)=\mathbb{P}\big(X(t)\le x\big),\qquad x\in R_t,\,t>0.\end{equation*} satisfies the relationship $F(x,t)=F_0(x,t)+F_1(x,t).$\\
\noindent By using standard techniques for the analysis of finite-velocity random motions it can be verified that, if the velocity $\mathtt{v}(\cdot)$ is sufficiently smooth, the following system of partial differential equations is satisfied
\begin{equation}
\begin{dcases}
\frac{\partial F_0}{\partial t}=-\mathtt{v}(x)\frac{\partial F_0}{\partial x}+\lambda F_1-\lambda F_0\\
\frac{\partial F_1}{\partial t}=\mathtt{v}(x)\frac{\partial F_1}{\partial x}+\lambda F_0-\lambda F_1.
\end{dcases}\label{Fsyscdf}
\end{equation}
For instance, the first equation of the system (\ref{Fsyscdf}) can be obtained by observing that
\begin{align}F_0(x,t+\Delta t)=F_0(x-\mathtt{v}(x)\Delta t,t)\,(1-\lambda \Delta t)+F_1(x,t)\lambda \Delta t + o(\Delta t).\label{toexpand}\end{align}
Taking a first-order Taylor expansion of equation (\ref{toexpand}), dividing by $\Delta t$ and taking the limit for $\Delta t\to0$ yields the desired equation. Similar arguments lead to the second equation of the system (\ref{Fsyscdf}). By defining now, for $j=0,1$, the probability density functions
\begin{equation*}f_j(x,t)=\mathbb{P}\big(X(t)\in\mathop{dx},\,D(t)=d_j\big)/\mathop{dx},\qquad x\in R_t,\,t>0\end{equation*}
and taking the derivatives with respect to $x$ of the equations in (\ref{Fsyscdf}), we obtain the system of partial differential equations
\begin{equation}
\begin{dcases}
\frac{\partial f_0}{\partial t}=-\frac{\partial}{\partial x}\big[\mathtt{v}(x)\,f_0\big]+\lambda f_1-\lambda f_0\\
\frac{\partial f_1}{\partial t}=\frac{\partial}{\partial x}\big[\mathtt{v}(x)\,f_1\big]+\lambda f_0-\lambda f_1.
\end{dcases}\label{fsyspdf}
\end{equation}
\noindent We now make some remarks. We observe that the system (\ref{fsyspdf}) is slightly different from that obtained in the papers by Masoliver and Weiss \cite{masoliverweiss} and Garra and Orsingher \cite{garraorsingher}. In particular, the velocity term $\mathtt{v}(x)$ and the differential operator $\frac{\partial}{\partial x}$ are interchanged. The aforementioned authors obtained, for the densities $f_0$ and $f_1$, a system of the form (\ref{Fsyscdf}), which we claim to be the system governing the cumulative distribution functions instead. Moreover, denoting by $$f(x,t)=\mathbb{P}\left(X(t)\in\,dx\right)/dx$$ the probability density function of $X(t)$ and taking into account that $f(x,t)=f_0(x,t)+f_1(x,t)$ the system (\ref{fsyspdf}) implies that the following partial differential equation is satisfied:
\begin{equation}
\frac{\partial^2 f}{\partial t^2}+2\lambda \frac{\partial f}{\partial t}=\frac{\partial}{\partial x} \left\{\mathtt{v}(x)\frac{\partial}{\partial x}\big[\mathtt{v}(x)\,f\big]\right\}.\label{fpderob}
\end{equation}
Again, equation (\ref{fpderob}) differs from that obtained by previous authors because the velocity term and the differential operator with respect to $x$ are interchanged. To confirm the correctness of equation (\ref{fpderob}), we use the fact that the process $X(t)$ can be viewed as a nonlinear transformation of a standard telegraph process, as shown by Masoliver and Weiss \cite{masoliverweiss} and Garra and Orsingher \cite{garraorsingher}. In particular, denoting by $\{\mathcal{T}(t)\}_{t\ge0}$ a standard telegraph process with constant velocity $c>0$, the following representation holds for $\mathcal{T}(t)$:
\begin{equation}\mathcal{T}(t)=c\int_{x_0}^{X(t)}\frac{1}{\mathtt{v}(w)}\,dw\label{TXrelat}\end{equation} provided that the integral in the right-hand side of the equation above is well-defined. To verify that the process $\mathcal{T}(t)$ defined by equation (\ref{TXrelat}) has constant velocity, it is sufficient to verify that $$\frac{d\mathcal{T}(t)}{dt}=c\,\frac{d X(t)}{dt}\,\frac{1}{\mathtt{v}(X(t))}=\pm c$$ where the sign depends on whether $D(t)=d_0$ or $D(t)=d_1$. Denote by $$P(z,t)=\mathbb{P}\left(\mathcal{T}(t)\le z\right),\qquad z\in[-ct,ct]$$ the cumulative distribution function of $\mathcal{T}(t)$ and by $p(z,t)$ the corresponding probability density function 
$$p(z,t)=\mathbb{P}\left(\mathcal{T}(t)\in dz\right)/\mathop{dz},\qquad z\in(-ct,ct).$$
Moreover, observe that, since $\mathtt{v}(\cdot)$ is positive-valued, the mapping $x\mapsto c\int_{x_0}^x\frac{1}{\mathtt{v}(w)}\,dw$ is strictly monotonically increasing and hence invertible. This implies that, in view of the relationship (\ref{TXrelat}) and setting $z=c\int_{x_0}^x\frac{1}{\mathtt{v}(w)}\,dw$, we can write \begin{equation*}F(x,t)=P(z,t)\end{equation*}
which in turn implies that \begin{equation*}\mathtt{v}(x)\, f(x,t)=c\;p(z,t).\end{equation*} It is now clear that, since $p(z,t)$ satisfies the partial differential equation
\begin{equation*}
\frac{\partial^2 p}{\partial t^2}+2\lambda \frac{\partial p}{\partial t}=c^2\,\frac{\partial^2 p}{\partial z^2}
\end{equation*}
the probability density function $f(x,t)$ must satisfy equation (\ref{fpderob}).\\

Up to this point, we have considered only the case in which the velocity $\mathtt{v}(\cdot)$ takes positive values in the support of the process $X(t)$. It is now interesting to observe that formula (\ref{TXrelat}) provides some insight into the behaviour of the process when the assumption of the velocity $\mathtt{v}(\cdot)$ being strictly positive is relaxed. Under the assumption that the velocity vanishes at some fixed point $b\in\mathbb{R}$, that is $\mathtt{v}(b)=0$, the question of whether the point $x=b$ is reachable by $X(t)$ naturally arises, since the velocity $\mathtt{v}(X(t))$ tends to zero as $X(t)\to b$. In view of formula (\ref{TXrelat}), the answer to this question depends on the rate at which the velocity vanishes in proximity of the point $b$. For instance, if the velocity vanishes at linear rate, that is if $$\lim_{x\to b}\frac{\mathtt{v}(x)}{\lvert x-b\lvert}=L$$ for some $L>0$, the point $x=b$ acts as an unreachable barrier, since it can only be reached in an infinite time. A notable case of this behaviour is that of linear velocity $\mathtt{v}(x)=c\,x$, with $c>0$. In this case, the point $x=0$ is unreachable and, if the starting point $x_0$ is strictly positive, the support of the process $X(t)$ is given by the positive half-line. This conclusion is clearly consistent with the representation \begin{equation}\label{intro:fvgbm}X(t)=x_0\,e^{\mathcal{T}(t)}\end{equation} which is obtained by formula (\ref{TXrelat}). Recalling that, in the hydrodynamic limit for $\lambda,c\to+\infty$ with $\frac{\lambda}{c^2}\to1$, the telegraph process $\mathcal{T}(t)$ converges in distribution to a standard Brownian motion, the process (\ref{intro:fvgbm}) can be regarded as a geometric telegraph process. This process has been investigated by Orsingher and De Gregorio \cite{degregorio} as the $y$-component of a hyperbolic Brownian motion on the Poincaré half-plane. The authors obtained the general expression of the moments of arbitrary order by using the characteristic function of $\mathcal{T}(t)$. The geometric telegraph process has also been investigated by Di Crescenzo and Pellerey \cite{dicrescenzo}, who discussed some financial applications related to option pricing. Moreover, the authors studied a generalization of the process in which the direction changes occur with Erlang-distributed intertimes.\\
\noindent On the other hand, if one assumes that $\mathtt{v}(\cdot)$ tends to zero with sublinear rate, that is if $$\lim_{x\to b}\frac{\mathtt{v}(x)}{\lvert x-b\lvert^\alpha}=L$$ for some $\alpha\in(0,1),\;L>0$, it can be verified that the point $b$ is reachable in finite time. This occurs because, since $\alpha<1$, the velocity decays slowly as $X(t)$ approaches $b$, and thus it does not prevent the process from reaching this point. In particular, assuming that the process $X(t)$ starts moving towards $b$ at time $t=0$ and does not change direction, the time $t^*$ required to reach the point $b$ is finite and reads
\begin{equation}t^*=\left\lvert\int_{x_0}^b\frac{1}{\mathtt{v}(w)}\,dw\right\lvert.\label{t*}\end{equation} In this case, finding the general distribution of $X(t)$ is beyond the aim of the present paper. We emphasize, however, that the relationship (\ref{toexpand}) and the related partial differential equation (\ref{fpderob}) can be proved to hold only for $t<t^*$. In other words, if the point $b$ is reachable in finite time, equation (\ref{fpderob}) fully characterizes the process only up to the time $t^*$. In fact, when the point $b$ is reached, a specific assumption on the behaviour of $X(t)$ is necessary in order for the process to be well-defined. For instance, one can assume that $b$ is an absorbing barrier or a reflecting barrier. Of course, these are not the only possibilities since alternative boundary conditions or additional stochastic mechanisms may also be introduced. Subsequently, the distribution of $X(t)$ for $t>t^*$ depends on the chosen assumption. We clarify this point with an example. Assume that the process $X(t)$ has space-varying velocity $$\mathtt{v}(x)=c\,x^{\alpha},\qquad \alpha\in(0,1),\;c>0.$$ and initial state $X(0)=x_0$, with $x_0>0$. The velocity function vanishes at zero with sublinear rate, implying that the point $b=0$ can be reached in a finite time $t^*$ given by $$t^*=\frac{x_0^{1-\alpha}}{c\,(1-\alpha)}.$$ For $t<t^*$, the time-dependent support $R_t$ of $X(t)$ reads $$R_t=\Big[\left(x_0^{1-\alpha}-c(1-\alpha)t\right)^{\frac{1}{1-\alpha}},\;\left(x_0^{1-\alpha}+c(1-\alpha)t\right)^{\frac{1}{1-\alpha}}\Big].$$ The distribution of $X(t)$ has a continuous component in the interior of $R_t$, which will henceforth be denoted as $\text{Int}(R_t)$. The extrema of the interval $R_t$ are reached with positive probability when no changes of direction occur. In view of formula (\ref{TXrelat}) and denoting by $\mathcal{T}(t)$ a standard telegraph process with velocity $c$, the process $X(t)$ admits the representation
\begin{equation}X(t)=\big[x_0^{1-\alpha}+(1-\alpha)\mathcal{T}(t)\big]^{\frac{1}{1-\alpha}},\qquad t<t^*.\label{alpharepr}\end{equation} Formula (\ref{alpharepr}) permits us to find the exact distribution of $X(t)$ for $t<t^*$. In particular, by recalling the probability density function of the telegraph process
\begin{align*}\mathbb{P}\left(\mathcal{T}(t)\in\mathop{dz}\right)/\mathop{dz}=\frac{e^{-\lambda t}}{2c}\Bigg[\lambda&\,I_0\left(\frac{\lambda}{c}\sqrt{c^2t^2-z^2}\right)\\&\qquad+\frac{\partial}{\partial t}I_0\left(\frac{\lambda}{c}\sqrt{c^2t^2-z^2}\right)\Bigg],\;\; \lvert z\lvert <t,\;t>0.\end{align*} it immediately follows that
\begin{align}\mathbb{P}\big(&X(t)\in\mathop{dx}\big)/\mathop{dx}=\frac{e^{-\lambda t}}{2cx^\alpha}\Bigg[\lambda\,I_0\left(\frac{\lambda}{c\,(1-\alpha)}\sqrt{c^2t^2(1-\alpha)^2-(x^{1-\alpha}-x_0^{1-\alpha})^2}\right)\nonumber\\
&+\frac{\partial}{\partial t}I_0\left(\frac{\lambda}{c\,(1-\alpha)}\sqrt{c^2t^2(1-\alpha)^2-(x^{1-\alpha}-x_0^{1-\alpha})^2}\right)\Bigg],\;\; x\in\text{Int}(R_t),\; t<t^*.\nonumber\end{align} We now discuss what happens for $t\ge t^*$. In principle, many different assumptions can be adopted on the behaviour of $X(t)$ for $t\ge t^*$. For instance, one can assume that the point $b=0$ acts as a reflecting barrier, in which case the process $X(t)$ can be expressed as
\begin{equation}X(t)=\big\lvert x_0^{1-\alpha}+(1-\alpha)\mathcal{T}(t)\big\lvert^{\frac{1}{1-\alpha}},\qquad t>0.\label{alphareprreflect}\end{equation}
Alternatively, it can be assumed that $b=0$ is an absorbing barrier. In this case, we can define the absorption time $\tau_0$ as $$\tau_0=\inf\left\{t>0:\;\mathcal{T}(t)\le -\,\frac{x_0^{1-\alpha}}{1-\alpha}\right\}$$ and write 
\begin{equation}X(t)=\big[ x_0^{1-\alpha}+(1-\alpha)\,\mathcal{T}(t \wedge\tau_0)\big]^{\frac{1}{1-\alpha}},\qquad t>0.\label{alphareprabsorb}\end{equation}
Clearly, in the cases (\ref{alphareprreflect}) and (\ref{alphareprabsorb}), finding the explicit distribution of $X(t)$ is not trivial for $t>t^*$. In particular, in the case (\ref{alphareprabsorb}), the joint distribution of $(\mathcal{T}(t),\,\tau_0)$ would be needed, which is very difficult to find (see Cinque and Orsingher \cite{cinque}). However, it is clear that both the representations (\ref{alphareprreflect}) and (\ref{alphareprabsorb}) coincide with (\ref{alpharepr}) when $t<t^*$, in which case equation (\ref{fpderob}) fully characterizes the distribution of $X(t)$.\\
\noindent Up to now, we have considered velocity functions $\mathtt{v}(\cdot)$ that vanish at a single point. In principle, it is possible to assume velocity functions that vanish at several points. As we will see, this yields finite-velocity processes which are naturally confined within bounded regions.

\section{On a bounded process with nonlinear velocity}
\noindent In this section, we study a univariate process $\{X(t)\}_{t\ge0}$ with initial value $x_0\in(0,1)$ and space-varying velocity \begin{equation}\label{x(1-x)}\mathtt{v}(x)=c\,x(1-x),\qquad x\in(0,1)\end{equation} with $c>0$. Clearly, the velocity (\ref{x(1-x)}) vanishes at the points $x=0$ and $x=1$. This naturally raises the question of whether a particle with position $X(t)$ can reach the points where the velocity vanishes. In view of the discussion of section \ref{prelimsec}, it turns out that both points $x=0$ and $x=1$ act as reflecting barriers and cannot be reached by the process, as the velocity $\mathtt{v}(\cdot)$ vanishes with linear rate at these points. A further verification of this fact can be achieved by simply observing that, in view of formula (\ref{TXrelat}), the process $X(t)$ admits the following representation:
\begin{equation}\label{TXx1-x}X(t)=\frac{x_0\,e^{\mathcal{T}(t)}}{1-x_0\,\left(1-e^{\mathcal{T}(t)}\right)}\end{equation}
where $\mathcal{T}(t)$ is a standard symmetric telegraph process with velocity $c$. Thus, the support of the process is given by the set
\begin{equation}\label{x(1-x)support}R_t=\left\{x\in\mathbb{R}:\;\frac{x_0\,e^{-ct}}{1-x_0\,\left(1-e^{-ct}\right)}\le x\le\frac{x_0\,e^{ct}}{1-x_0\,\left(1-e^{ct}\right)}\right\}\end{equation} where the extrema of the interval can be reached if and only if the process $X(t)$ performs no changes of direction until time $t$. In order to study the continuous component of the distribution of $X(t)$ in the interior of the support (\ref{x(1-x)support}), we employ the relationship (\ref{TXx1-x}) which implies that, for $t>0$ and $x\in\left(\frac{x_0\,e^{-ct}}{1-x_0\,\left(1-e^{-ct}\right)},\;\frac{x_0\,e^{ct}}{1-x_0\,\left(1-e^{ct}\right)}\right)$, the probability density function $f$ of $X(t)$ can be expressed explicitly in the form
\begin{align}f(x,t)=&\frac{e^{-\lambda t}}{2c\,x(1-x)}\Bigg\{\lambda\, I_0\left(\frac{\lambda}{c}\,\sqrt{c^2t^2-\log^2\left(\frac{x\,(1-x_0)}{x_0(1-x)}\right)}\right)\nonumber\\&\qquad+I_0\left(\frac{\lambda}{c}\,\sqrt{c^2t^2-\log^2\left(\frac{x\,(1-x_0)}{x_0(1-x)}\right)}\right)\Bigg\}.\nonumber\end{align}
We are now interested in studying the moments of the process $X(t)$. Clearly, due to the nonlinearity of the relationship between $X(t)$ and $\mathcal{T}(t)$, finding an explicit expression for the moments is, in general, a difficult task. In the special case $x_0=\frac{1}{2}$, the symmetry of the distribution can be exploited to obtain the following result.
\begin{thm}For $t>0$, it holds that $$\mathbb{E}\left[\frac{e^{\mathcal{T}(t)}}{1+e^{\mathcal{T}(t)}}\right]=\frac{1}{2}.$$\label{thm:symthm}\end{thm}
\begin{proof}We have that \begin{align}\mathbb{E}\left[\frac{e^{\mathcal{T}(t)}}{1+e^{\mathcal{T}(t)}}\right]=&\mathbb{E}\left[\frac{e^{\mathcal{T}(t)}}{1+e^{\mathcal{T}(t)}}\;\Big\lvert\; \mathcal{T}(t)>0\right]\cdot \mathbb{P}\left(\mathcal{T}(t)>0\right)\nonumber\\
&\qquad\qquad+\mathbb{E}\left[\frac{e^{\mathcal{T}(t)}}{1+e^{\mathcal{T}(t)}}\;\Big\lvert\; \mathcal{T}(t)<0\right]\cdot \mathbb{P}\left(\mathcal{T}(t)<0\right)\nonumber\\=&\frac{1}{2}\,\mathbb{E}\left[\frac{e^{\mathcal{T}(t)}}{1+e^{\mathcal{T}(t)}}\;\Big\lvert\; \mathcal{T}(t)>0\right]+\frac{1}{2}\,\mathbb{E}\left[\frac{e^{\mathcal{T}(t)}}{1+e^{\mathcal{T}(t)}}\;\Big\lvert\; \mathcal{T}(t)<0\right].\nonumber\end{align}We now observe that, since $\mathcal{T}(t)$ is symmetric, it is identical in distribution to $-\mathcal{T}(t)$. Therefore, we can write
\begin{align}\mathbb{E}\left[\frac{e^{\mathcal{T}(t)}}{1+e^{\mathcal{T}(t)}}\right]=&\frac{1}{2}\,\mathbb{E}\left[\frac{e^{\mathcal{T}(t)}}{1+e^{\mathcal{T}(t)}}\;\Big\lvert\; \mathcal{T}(t)>0\right]+\frac{1}{2}\,\mathbb{E}\left[\frac{e^{-\mathcal{T}(t)}}{1+e^{-\mathcal{T}(t)}}\;\Big\lvert\; \mathcal{T}(t)>0\right]\nonumber\\
=&\frac{1}{2}\,\mathbb{E}\left[1\;\Big\lvert\; \mathcal{T}(t)>0\right]=\frac{1}{2}\nonumber\end{align} which completes the proof.
\end{proof}
To extend the result to the general case $x_0\neq\frac{1}{2}$, we seek to express the process $X(t)$ as a power series in terms of $\mathcal{T}(t)$. To this end, we give the following preliminary result.
\begin{thm}\label{thm:poly}
The following Taylor series expansion holds for $a,\theta>0$:
\begin{equation}\left(\frac{\theta+1}{\theta e^t+1}\right)^ae^{xt}=\sum_{n=0}^{+\infty}\frac{E_n^{(a,\theta)}(x)}{n!}\,t^n,\qquad t,x\in\mathbb{R}\label{taylorseriesstatement}\end{equation}
where
\begin{equation}E_n^{(a,\theta)}(x)=\sum_{k=0}^n\binom{n}{k}x^{n-k}\sum_{j=0}^k\binom{-a}{j} j!\;\left\{k \atop j\right\}\left(\frac{\theta}{\theta+1}\right)^j\label{Eulerpolydef}\end{equation}
and the coefficients $\left\{k \atop j\right\},\; j=0,...,k,$ represent the Stirling numbers of the second kind.
\end{thm}
\begin{proof}We set $$h(t)=\left(\frac{\theta+1}{\theta e^t+1}\right)^ae^{xt}$$ and we denote by $h^{(n)}(t)$ the $n$-th order derivative of $h(t)$ with respect to $t$. Since the Taylor series expansion of $h(t)$ reads $$h(t)=\sum_{n=0}^{+\infty}\frac{h^{(n)}(0)}{n!}\,t^n$$ it is clear that, in order to prove formula (\ref{taylorseriesstatement}), we need to show that \begin{equation}\label{coefficientstoprove}h^{(n)}(0)=E_n^{(a,\theta)}(x).\end{equation} We start by observing that
\begin{align}h^{(n)}(t)=&(\theta+1)^a\frac{d^n}{dt^n}\left[e^{xt}\;(\theta e^t+1)^{-a}\right]=(\theta+1)^a\sum_{k=0}^n\binom{n}{k}\frac{d^{n-k} e^{xt}}{dt^{n-k}}\;\frac{d^k}{dt^k}(\theta e^t+1)^{-a}\nonumber\\
=&(\theta+1)^a\,e^{xt}\sum_{k=0}^n\binom{n}{k} x^{n-k}\;\frac{d^k}{dt^k}(\theta e^t+1)^{-a}.\label{tempeuler}\end{align}
We now claim that \begin{equation}\frac{d^k}{dt^k}(\theta e^t+1)^{-a}=\frac{1}{\left(\theta e^t+1\right)^a}\sum_{j=0}^k\binom{-a}{j}j!\;\left\{k \atop j\right\}\,\left(\frac{\theta e^{t}}{\theta e^t+1}\right)^{j}.\label{toprovebyinduction}\end{equation} Formula (\ref{toprovebyinduction}) is trivial for $k=0$ and can be proved by induction for $k\ge1$. Indeed, the induction principle yields 
\begin{align}\frac{d^k}{dt^k}(\theta& e^t+1)^{-a}=\frac{d}{dt}\left[\frac{d^{k-1}}{dt^{k-1}}(\theta e^t+1)^{-a}\right]\nonumber\\
=&\sum_{j=0}^{k-1}\binom{-a}{j} j!\;\left\{k-1 \atop j\right\}\,\theta^j\,\frac{d}{dt}\left[e^{jt}\,(\theta e^t+1)^{-j-a}\right]\nonumber\\
=&\frac{1}{\left(\theta e^t+1\right)^a}\sum_{j=1}^{k-1}\binom{-a}{j}j!\;j\left\{k-1 \atop j\right\}\,\left(\frac{\theta e^{t}}{\theta e^t+1}\right)^{j}\nonumber\\&\qquad+\frac{1}{\left(\theta e^t+1\right)^a}\sum_{j=0}^{k-1}\binom{-a}{j}(-j-a)\, j!\;\left\{k-1 \atop j\right\}\,\left(\frac{\theta e^{t}}{\theta e^t+1}\right)^{j+1}\nonumber\\
=&\frac{1}{\left(\theta e^t+1\right)^a}\sum_{j=1}^{k-1}\binom{-a}{j}j!\;j\left\{k-1 \atop j\right\}\,\left(\frac{\theta e^{t}}{\theta e^t+1}\right)^{j}\nonumber\\&\qquad+\frac{1}{\left(\theta e^t+1\right)^a}\sum_{j=0}^{k-1}\binom{-a}{j+1}\, (j+1)!\;\left\{k-1 \atop j\right\}\,\left(\frac{\theta e^{t}}{\theta e^t+1}\right)^{j+1}\nonumber\\
=&\binom{-a}{k}\frac{k!}{\left(\theta e^t+1\right)^a}\left(\frac{\theta e^{t}}{\theta e^t+1}\right)^{k}\nonumber\\
&\qquad+\frac{1}{\left(\theta e^t+1\right)^a}\sum_{j=1}^{k-1}\binom{-a}{j}j!\;\left(j\left\{k-1 \atop j\right\}+\left\{k-1 \atop j-1\right\}\right)\,\left(\frac{\theta e^{t}}{\theta e^t+1}\right)^{j}.\label{tempinduction}
\end{align}
By using the recursive formula $$\left\{k \atop j\right\}=j\left\{k-1 \atop j\right\}+\left\{k-1 \atop j-1\right\},\qquad0<j<k$$ and taking into account that $\left\{k \atop 0\right\}=0$ and $\left\{k \atop k\right\}=1$, formula (\ref{tempinduction}) immediately proves equation (\ref{toprovebyinduction}). By substituting now formula (\ref{toprovebyinduction}) into (\ref{tempeuler}), we obtain \begin{equation}h^{(n)}(t)=\left(\frac{\theta+1}{\theta e^t+1}\right)^ae^{xt}\sum_{k=0}^n\binom{n}{k} x^{n-k}\sum_{j=0}^k\binom{-a}{j}j!\;\left\{k \atop j\right\}\,\left(\frac{\theta e^{t}}{\theta e^t+1}\right)^{j}.\label{h(t)derivative}\end{equation}
 Setting $t=0$ in equation (\ref{h(t)derivative}) finally yields formula (\ref{coefficientstoprove}) which completes the proof of (\ref{taylorseriesstatement}).
\end{proof}
The polynomials $E_n^{(a,\theta)}(x)$ introduced in theorem \ref{thm:poly} provide an interesting generalization of the well-known Euler polynomials $E_n(x),\, n\in\mathbb{N}$, (see Gradshteyn and Ryzhik \cite{gradshteyn}), which naturally emerge from the generating function 
$$\frac{2e^{xt}}{e^t+1}=\sum_{n=0}^{+\infty}\frac{E_n(x)}{n!}\,t^n.$$ We emphasize that an extension of the Euler polynomials involving the parameter $a$ has already been studied in the literature (see Roman \cite{roman}). In theorem \ref{thm:poly}, we further generalized this family of polynomials by introducing the additional parameter $\theta$. Similarly to the classical Euler polynomials, the polynomials $E_n^{(a,\theta)}(x)$ have some interesting properties from the point of view of umbral calculus. Indeed, they form an Appell sequence, meaning that they satisfy the relationship
$$\frac{d}{dx}E_n^{(a,\theta)}(x)=n\,E_{n-1}^{(a,\theta)}(x).$$
From our perspective, however, these polynomials provide an explicit expression for the moments of $X(t)$. We first observe that the process $X(t)$ can be represented as
$$X(t)=x_0\sum_{n=0}^{+\infty}\frac{E_n^{\left(1,\frac{x_0}{1-x_0}\right)}(1)}{n!}\;\mathcal{T}(t)^n$$
and, in general, for $a>0$,
\begin{equation}X(t)^a=x_0^a\sum_{n=0}^{+\infty}\frac{E_n^{\left(a,\frac{x_0}{1-x_0}\right)}(a)}{n!}\;\mathcal{T}(t)^n\label{Xpowerseriesa}\end{equation}
We then recall that the odd-order moments of the telegraph process $\mathcal{T}(t)$ vanish, while the even-order moments read (see Iacus and Yoshida \cite{iacus})
\begin{equation}\mathbb{E}\left[\mathcal{T}(t)^{2n}\right]=e^{-\lambda t}(ct)^{2n}\left(\frac{2}{\lambda t}\right)^{n-\frac{1}{2}}\Gamma\left(n+\frac{1}{2}\right)\left[I_{n+\frac{1}{2}}(\lambda t)+I_{n-\frac{1}{2}}(\lambda t)\right].\label{iacusmoments}\end{equation}
Hence, by combining formulas (\ref{Xpowerseriesa}) and (\ref{iacusmoments}), we obtain that
\begin{align}\mathbb{E}\left[X(t)^a\right]&=x_0^a\sum_{n=0}^{+\infty}\frac{E_{2n}^{\left(a,\frac{x_0}{1-x_0}\right)}(a)}{(2n)!}\;\mathbb{E}\left[\mathcal{T}(t)^{2n}\right]\nonumber\\
&=x_0^a\;\sqrt{\frac{\lambda t}{2}}\,e^{-\lambda t}\sum_{n=0}^{+\infty}\frac{\Gamma\left(n+\frac{1}{2}\right)\;E_{2n}^{\left(a,\frac{x_0}{1-x_0}\right)}(a)}{(2n)!}\cdot\nonumber\\&\qquad\qquad\qquad\qquad\qquad\qquad\cdot\left[I_{n+\frac{1}{2}}(\lambda t)+I_{n-\frac{1}{2}}(\lambda t)\right]\left(\frac{2c^2t}{\lambda}\right)^n.\label{Xpowerseriesmoments}\end{align}
We emphasize that, in formula (\ref{Xpowerseriesmoments}), the expected value and summation symbols were interchanged. This step can be justified by using the dominated convergence theorem. Moreover, we observe that formula (\ref{Xpowerseriesmoments}) is consistent with theorem \ref{thm:symthm}. In fact, by setting $a=1$ and $x_0=\frac{1}{2}$, the generalized Euler polynomials in formula (\ref{Xpowerseriesmoments}) reduce to the classical Euler polynomials. Hence, taking into account that $$E_{2n}(1)=\begin{cases}1\qquad \text{if}\;n=0\\0\qquad \text{if}\;n\ge0\end{cases}$$ and recalling that $$I_{\frac{1}{2}}(x)=\sqrt{\frac{2}{\pi x}\,}\;\sinh(x),\qquad\quad I_{-\frac{1}{2}}(x)=\sqrt{\frac{2}{\pi x}\,}\;\cosh(x)$$ formula (\ref{Xpowerseriesmoments}) clearly yields $\mathbb{E}\left[X(t)\right]=\frac{1}{2}.$\\
\noindent We are now interested in studying the hydrodynamic limit of the process $X(t)$ for $\lambda,c\to+\infty$, with $\frac{\lambda}{c^2}\to1$. To this aim, we first observe that, in view of equation (\ref{fpderob}), the probability density function $f$ of $X(t)$ is a solution to the partial differential equation
\begin{equation}
\frac{\partial^2 f}{\partial t^2}+2\lambda \frac{\partial f}{\partial t}=c^2 \frac{\partial}{\partial x} \left\{x(1-x)\frac{\partial}{\partial x}\big[x(1-x)\,f\big]\right\},\qquad x\in\text{Int}(R_t),\;t>0.\label{f_x(1-x)_pde}
\end{equation}
By dividing equation (\ref{f_x(1-x)_pde}) by $2\lambda$, and by taking the hydrodynamic limit, we obtain the diffusion equation
\begin{equation}
\frac{\partial f}{\partial t}=\frac{1}{2} \frac{\partial}{\partial x} \left\{x(1-x)\frac{\partial}{\partial x}\big[x(1-x)\,f\big]\right\},\qquad x\in(0,1),\;t>0.\label{x(1-x)_limit_pde}
\end{equation}
which coincides with the Fokker-Planck equation of the process satisfying the It\^{o} stochastic differential equation
\begin{equation}dX(t)=\frac{1}{2}\,X(t)(1-X(t))(1-2X(t))dt+X(t)(1-X(t))\,dB(t)\label{Xsedtren}\end{equation} where $\{B(t)\}_{t\ge0}$ is a standard Brownian motion. Hence, by solving the stochastic differential equation (\ref{Xsedtren}), it can be verified that \begin{equation}\label{BXx1-x}\lim_{\lambda,c\to+\infty}X(t)\overset{i.d.}{=}\frac{x_0\,e^{B(t)}}{1-x_0\,\left(1-e^{B(t)}\right)}.\end{equation}
Observe that, in view of the well-known limiting distribution of the standard telegraph process, formula (\ref{BXx1-x}) is consistent with (\ref{TXx1-x}). By taking now the hydrodynamic limit of formula (\ref{BXx1-x}), and recalling the asymptotic behaviour of the modified Bessel function for large argument $z\in\mathbb{R}^+$ $$I_\alpha(z)\sim\frac{e^z}{\sqrt{2\pi z}},\qquad \forall \alpha\in\mathbb{R}$$
we obtain that
\begin{equation*}\mathbb{E}\left[\left(\frac{x_0\,e^{B(t)}}{1-x_0\,\left(1-e^{B(t)}\right)}\right)^a\right]=\frac{x_0^a}{\sqrt{\pi}}\sum_{n=0}^{+\infty}\frac{\Gamma\left(n+\frac{1}{2}\right)\;E_{2n}^{\left(a,\frac{x_0}{1-x_0}\right)}(a)}{(2n)!}\;\left(2t\right)^n.\end{equation*}

\section{Multivariate extensions of motions with space-varying velocities}\label{sec:multivariate}
\noindent In this section, we study a planar random motion with space-varying velocity. We consider a bivariate stochastic process $(X(t),\,Y(t))$ which can move along four possible directions $d_j,\;j=0,1,2,3$. The directions are assumed to be orthogonal and are defined as
$$d_j=\left(\cos\left(\frac{\pi j}{2}\right),\;\sin\left(\frac{\pi j}{2}\right)\right)$$ where, of course, $d_j=d_{j+4n}$ for all integers $n$. At the initial time $t=0$, the process lies at the point $(x_0,y_0)$ and starts moving along one of the four possible directions, chosen randomly with equal probability $\frac{1}{4}$. As usual, the direction changes are governed by a Poisson process $N(t)$ with constant intensity $\lambda$. When a Poisson event occurs, the process $(X(t),\,Y(t))$ can turn either clockwise, from $d_j$ to $d_{j-1}$, or counterclockwise, from $d_j$ to $d_{j+1}$, each with probability $\frac{1}{2}$. The direction of $(X(t),\,Y(t))$ at time $t$ is denoted by $D(t)$.\\
We assume that the velocity of the process depends on its position along the direction of motion. Hence, when the process $(X(t),\,Y(t))$ is moving horizontally, its velocity depends on the abscissa $X(t)$. In particular, for a fixed positive-valued function $\mathtt{v}(\cdot)$, the velocity of the process is assumed to be $\mathtt{v}(X(t))$. Similarly, when the process is moving vertically, its velocity is given by $\mathtt{v}(Y(t))$. As usual, we assume that the function $\mathtt{v}(\cdot)$ is continuously differentiable. Moreover, we emphasize that the time-dependent support $R_t$ of the process depends on the functional form of the velocity.\\
We are interested in studying the probability density function
$$f(x,y,t)\mathop{dx}\mathop{dy}=\mathbb{P}\big(X(t)\in\mathop{dx},\;Y(t)\in\mathop{dy}\big), \qquad\quad(x,y)\in R_t,\; t>0.$$
\noindent For this purpose, we consider, for $j=0,1,2,3$, the auxiliary density functions
$$f_j(x,y,t)\mathop{dx}\mathop{dy}=\mathbb{P}\big(X(t)\in\mathop{dx},\;Y(t)\in\mathop{dy},\; D(t)=d_j\big), \qquad(x,y)\in R_t,\; t>0.$$
By using the methods described in section \ref{prelimsec}, it can be proved that the following linear system of partial differential equations holds
\begin{equation}
\begin{dcases}
\frac{\partial f_0}{\partial t}=-\frac{\partial}{\partial x}\big[\mathtt{v}(x)\,f_0\big]+\frac{\lambda}{2} f_1+\frac{\lambda}{2} f_3-\lambda f_0\\
\frac{\partial f_1}{\partial t}=-\frac{\partial}{\partial y}\big[\mathtt{v}(y)\,f_1\big]+\frac{\lambda}{2} f_0+\frac{\lambda}{2} f_2-\lambda f_1\\
\frac{\partial f_2}{\partial t}=\frac{\partial}{\partial x}\big[\mathtt{v}(x)\,f_2\big]+\frac{\lambda}{2} f_1+\frac{\lambda}{2} f_3-\lambda f_2\\
\frac{\partial f_3}{\partial t}=\frac{\partial}{\partial y}\big[\mathtt{v}(y)\,f_3\big]+\frac{\lambda}{2} f_0+\frac{\lambda}{2} f_2-\lambda f_3.
\end{dcases}\label{usualplanar}
\end{equation}
In order to study the distribution of $(X(t),\,Y(t))$, we define a stochastic process $(U(t),\,V(t))$ as \begin{equation}\label{UVdef}\begin{cases}U(t)=c\,\int_{x_0}^{X(t)}\frac{1}{\mathtt{v}(z)}\,dz\\[1.2ex] V(t)=c\,\int_{y_0}^{Y(t)}\frac{1}{\mathtt{v}(z)}\,dz\end{cases}\end{equation} provided that the integrals in formula (\ref{UVdef}) exist. The constant $c$ is assumed to be strictly positive. We denote by $S_t$ the support of $(U(t),\,V(t))$, and by $g$ its probability density function
$$g(u,v,t)\mathop{du}\mathop{dv}=\mathbb{P}\big(U(t)\in\mathop{du},\;V(t)\in\mathop{dv}\big), \qquad\quad(u,v)\in S_t,\; t>0.$$
Moreover, we denote by $g_j,\;j=0,1,2,3$, the density functions
$$g_j(u,v,t)\mathop{du}\mathop{dv}=\mathbb{P}\big(U(t)\in\mathop{du},\;V(t)\in\mathop{dv},\; D(t)=d_j\big), \qquad(u,v)\in S_t,\; t>0.$$
We now observe that the transformation (\ref{UVdef}) implies that, by setting $u=\int_{x_0}^x\frac{c}{\mathtt{v}(w)}\,dw$ and $v=\int_{y_0}^y\frac{c}{\mathtt{v}(w)}\,dw$, the following relationship holds for $j=0,1,2,3$
\begin{equation}c^2\;g_j(u,v,t)=f_j(x,y,t)\,\mathtt{v}(x)\,\mathtt{v}(y).\label{bivchange}\end{equation}
Substituting formula (\ref{bivchange}) into the system (\ref{usualplanar}) immediately yields that 
\begin{equation}
\begin{dcases}
\frac{\partial g_0}{\partial t}=-c\,\frac{\partial g_0}{\partial u}+\frac{\lambda}{2} g_1+\frac{\lambda}{2} g_3-\lambda g_0\\
\frac{\partial g_1}{\partial t}=-c\,\frac{\partial g_1}{\partial v}+\frac{\lambda}{2} g_0+\frac{\lambda}{2} g_2-\lambda g_1\\
\frac{\partial g_2}{\partial t}=c\,\frac{\partial g_2}{\partial u}+\frac{\lambda}{2} g_1+\frac{\lambda}{2} g_3-\lambda g_2\\
\frac{\partial g_3}{\partial t}=c\,\frac{\partial g_3}{\partial v}+\frac{\lambda}{2} g_0+\frac{\lambda}{2} g_2-\lambda g_3.
\end{dcases}\label{constantplanar}
\end{equation}
The system (\ref{constantplanar}) is well-known in the literature, as it describes the distribution of a planar random motion with orthogonal directions and constant velocity (see Orsingher and Marchione \cite{orsinghermarchione}). In other words, the process $(U(t),\,V(t))$ defined by the transformation (\ref{UVdef}) is a finite-velocity random motion which moves along the directions $d_j,\;j=0,1,2,3$, changes direction at times governed by the Poisson process $N(t)$ and travels at constant velocity $c$. The exact distribution of this process has been studied by several authors. The support of the process is given by the square \begin{equation}\label{supportgenerali}S_t=\{(u,v)\in\mathbb{R}^2:\;\;\lvert u+v\lvert\le ct,\;\lvert u-v\lvert\le ct\}\end{equation} and the continuous component of the distribution in the interior of the square reads
\begin{align}g(u,v,t)=\frac{e^{-\lambda t}}{2c^2}&\bigg[\frac{\lambda}{2}\, I_0\bigg(\frac{\lambda}{2c}\sqrt{c^2t^2-(u+v)^2}\bigg)+\frac{\partial}{\partial t}I_0\bigg(\frac{\lambda}{2c}\sqrt{c^2t^2-(u+v)^2}\bigg)\bigg]\nonumber\\
\cdot&\left[\frac{\lambda}{2}\, I_0\left(\frac{\lambda}{2c}\sqrt{c^2t^2-(u-v)^2}\right)+\frac{\partial}{\partial t}I_0\left(\frac{\lambda}{2c}\sqrt{c^2t^2-(u-v)^2}\right)\right].\label{pexactbiv}\end{align}
The distribution (\ref{pexactbiv}) was first obtained by Orsingher \cite{orsingher2000exact}, and it was proved to hold also in the case of a non-homogeneous Poisson process by Cinque and Orsingher \cite{cinquenonhom}. We now observe that, since formula (\ref{UVdef}) implies that the vector $(X(t),\,Y(t))$ can be expressed as a nonlinear transformation of $(U(t),\,V(t))$, we are now able to obtain the exact distribution of $(X(t),\,Y(t))$ in the interior of $R_t$.
\begin{thm}
The probability density function of the process $(X(t),\,Y(t))$ in the interior of $R_t$ reads
\begin{align}&f(x,y,t)=\frac{e^{-\lambda t}}{2\,\mathtt{v}(x)\,\mathtt{v}(y)}\cdot\Bigg[\frac{\lambda}{2}\, I_0\Bigg(\frac{\lambda}{2}\sqrt{t^2-\left(\int_{x_0}^x\frac{dz}{\mathtt{v}(z)}+\int_{y_0}^y\frac{dz}{\mathtt{v}(z)}\right)^2}\;\Bigg)\nonumber\\
&\qquad\qquad+\frac{\partial}{\partial t}I_0\Bigg(\frac{\lambda}{2}\sqrt{t^2-\left(\int_{x_0}^x\frac{dz}{\mathtt{v}(z)}+\int_{y_0}^y\frac{dz}{\mathtt{v}(z)}\right)^2}\;\Bigg)\Bigg]\nonumber\\
&\qquad\qquad\quad\cdot\Bigg[\frac{\lambda}{2}\, I_0\Bigg(\frac{\lambda}{2}\sqrt{t^2-\left(\int_{x_0}^x\frac{dz}{\mathtt{v}(z)}-\int_{y_0}^y\frac{dz}{\mathtt{v}(z)}\right)^2}\Bigg)\nonumber\\
&\qquad\qquad\qquad+\frac{\partial}{\partial t}I_0\Bigg(\frac{\lambda}{2}\sqrt{t^2-\left(\int_{x_0}^x\frac{dz}{\mathtt{v}(z)}-\int_{y_0}^y\frac{dz}{\mathtt{v}(z)}\right)^2}\Bigg)\Bigg],\; (x,y)\in \normalfont{\text{Int}}(R_t).\nonumber\end{align}\end{thm}
\begin{proof}The result follows immediately by combining formulas (\ref{bivchange}) and (\ref{pexactbiv}).\end{proof}

We now remark that the distribution of $(X(t),\,Y(t))$ exhibits a singular component on the boundary of the support $\partial R_t$. In particular, by means of usual combinatorial techniques, it can be shown that $$\mathbb{P}\big((X(t),\,Y(t))\in\partial R_t\big)=2e^{-\frac{\lambda t}{2}}-e^{-\lambda t}.$$ Hence, for a complete analysis of the process, the exact distribution of $(X(t),\,Y(t))$ on the boundary $\partial R_t$ must be investigated. To this end, we begin by examining the general case for an arbitrary velocity function $\mathtt{v}(\cdot)$. Subsequently, we will consider a specific choice of $\mathtt{v}(\cdot)$ to illustrate how the expressions simplify and to gain geometric insight into the behavior of the process.\\
\noindent Recall that the support $R_t$ of the process $(X(t),\,Y(t))$ and the support of $(U(t),\,V(t))$, that is the square $S_t$, can be mapped onto each other by means of the transformation (\ref{UVdef}). Moreover, the method typically used for studying the distribution of $(U(t),\,V(t))$ on the boundary $\partial S_t$ of its support is to restrict the analysis to the side of the square which belongs to the first quadrant of the Cartesian plane, that is the set
 $$S^1_t=\{(u,v)\in\mathbb{R}^2: \;\;v=ct-u,\, 0<u<ct\}.$$
The distribution on the remaining sides can be studied analogously. The same principle applies in the case of space-varying velocity. Therefore, we restrict the analysis to the portion of the boundary $\partial R_t$ which belongs to the first quadrant:
$$R^1_t=\left\{(x,y)\in\mathbb{R}^2:\;\;\int_{x_0}^x\frac{1}{\mathtt{v}(z)}dz+\int_{y_0}^y\frac{1}{\mathtt{v}(z)}dz=t,\;\, 0<\int_{x_0}^x\frac{1}{\mathtt{v}(z)}dz<t\right\}.$$
Hence, to study the distribution of $(X(t),\,Y(t))$ on $\partial R_t$, we consider, without loss of generality, the side $R^1_t$. We propose two different ways to study the behaviour of the process on $R^1_t$. The first approach is based on the construction of a suitable coordinate to identify the position of $(X(t),\,Y(t))$ on $R^1_t$, namely $$H(t)=\int_{x_0}^{X(t)}\frac{c}{\mathtt{v}(z)}dz-\int_{y_0}^{Y(t)}\frac{c}{\mathtt{v}(z)}dz.$$
It can be verified that $H(t)\in[-ct,ct]$. Using the coordinate $H(t)$ permits us to employ well-known results concerning planar random motions with constant velocity. In particular, it holds that
\begin{align}\mathbb{P}\Big((X(t),\,&Y(t))\in R^1_t,\;H(t)\in d\eta\Big)/d\eta\nonumber\\
=&\mathbb{P}\Big(U(t)+V(t)=ct,\;U(t)-V(t)\in d\eta\Big)/d\eta\nonumber\\
=&\frac{e^{-\lambda t}}{4c}\left[\frac{\lambda}{2}\;I_0\left(\frac{\lambda}{2c}\sqrt{c^2t^2-\eta^2}\right)+\;\frac{\partial}{\partial t}I_0\left(\frac{\lambda}{2c}\sqrt{c^2t^2-\eta^2}\right)\right],\qquad \lvert\eta\lvert<ct.\label{Hdistribu}\end{align}
We refer to formula (17) of the paper by Marchione and Orsingher \cite{marchione} for details. However, since we find the coordinate $H(t)$ not inuitive, we introduce an alternative way to study the distribution of $(X(t),\,Y(t))$ on $R^1_t$. In particular, we propose to identify the position of the process on $R^1_t$ through its abscissa $X(t)$. We begin our analysis by examining the simpler case of constant velocity $c$. Hence, we study the probability density function \begin{equation}\label{ialen}q(u,t)=\mathbb{P}\big(U(t)+V(t)=ct,\;U(t)\in du\big)/du,\qquad \lvert u \lvert<ct,\;t>0.\end{equation}
We first define, for $j=0,1$, the auxiliary density functions
\begin{equation}\label{qdefinition}q_j(u,t)=\mathbb{P}\big(U(t)+V(t)=ct,\;U(t)\in du,\;D(t)=d_j\big)/du,\qquad \lvert u \lvert<ct,\;t>0.\end{equation}
By means of standard methods, it can be shown that the following system of equations holds
\begin{equation*}
\begin{cases}
q_0(u,t+\Delta t)=q_0(u-c\Delta t,t)\,(1-\lambda \Delta t)+q_1(u,t)\,\frac{\lambda}{2}\Delta t+o(\Delta t)\\
q_1(u,t+\Delta t)=q_1(u,t)\,(1-\lambda \Delta t)+q_0(u-c\Delta t,t)\,\frac{\lambda}{2}\Delta t+o(\Delta t)
\end{cases}
\end{equation*}
which implies that
\begin{equation}\label{UVsys}
\begin{dcases}
\frac{\partial q_0}{\partial t}=-c\frac{\partial q_0}{\partial u}+\frac{\lambda}{2}\,q_1-\lambda\,q_0\\
\frac{\partial q_1}{\partial t}=\frac{\lambda}{2}\,q_0-\lambda\,q_1.
\end{dcases}
\end{equation}
In view of the system (\ref{UVsys}), the probability density function (\ref{qdefinition}) satisfies the second-order partial differential equation
\begin{equation}\label{pdeqlong}\left(\frac{\partial^2}{\partial t^2}+c\frac{\partial^2}{\partial u\partial t}+2\lambda\frac{\partial}{\partial t}+c\lambda\frac{\partial}{\partial u}+\frac{3\lambda^2}{4}\right)\,q=0.\end{equation} The density function $q$ emerging from equation (\ref{pdeqlong}) is given in the following theorem.
\begin{thm}The probability density function (\ref{ialen}) reads
\begin{equation}\label{q_distr}q(u,t)=\frac{e^{-\lambda t}}{2c}\left[\frac{\lambda}{2}\,I_0\left(\frac{\lambda}{c}\,\sqrt{u\,(ct-u)}\right)+\frac{\partial}{\partial t}\,I_0\left(\frac{\lambda}{c}\,\sqrt{u\,(ct-u)}\right)\right],\qquad 0<u<ct.\end{equation}
\end{thm}
\begin{proof}
In order to obtain (\ref{q_distr}) from (\ref{pdeqlong}), we define the function \begin{equation}\label{qtildedef}\widetilde{q}(u,t)=e^{\lambda t}\,q(u,t)\end{equation} which satisfies the partial differential equation
\begin{equation}\label{pdeq}\left(\frac{\partial^2}{\partial t^2}+c\frac{\partial^2}{\partial u\partial t}-\frac{\lambda^2}{4}\right)\,\widetilde{q}=0.\end{equation}
By performing the change of variables \begin{equation}\label{zdef}z=\sqrt{u\,(ct-u)}\end{equation}
equation (\ref{pdeq}) reduces to the Bessel equation \begin{equation}\label{Besseleq}\frac{d^2\widetilde{q}}{dz^2}+\frac{1}{z}\,\frac{d\,\widetilde{q}}{dz}-\frac{\lambda^2}{c^2}\,\widetilde{q}=0.\end{equation} The general solution to equation (\ref{Besseleq}) is $$\widetilde{q}(z)=A\,I_0\left(\frac{\lambda}{c}\,z\right)+B\,K_0\left(\frac{\lambda}{c}\,z\right)$$ where $I_0(\cdot)$ and $K_0(\cdot)$ represent the modified Bessel functions of the first and the second kind respectively. However, we disregard the term involving $K_0(\cdot)$ since it would make the final solution non-integrable. Thus, by inverting the transformations (\ref{zdef}) and (\ref{qtildedef}), we express the general solution to equation (\ref{pdeqlong}) in the form
$$q(u,t)=e^{-\lambda t}\left[A\,I_0\left(\frac{\lambda}{c}\,\sqrt{u\,(ct-u)}\right)+B\frac{\partial}{\partial t}\,I_0\left(\frac{\lambda}{c}\,\sqrt{u\,(ct-u)}\right)\right].$$ Observe that we have included an additional term involving the time-derivative of the Bessel function in order to introduce greater flexibility into the general solution. This poses no issue as the Bessel function is a solution to  (\ref{pdeqlong}), and thus its derivative must be a solution too. This immediately follows by the fact that equation (\ref{pdeq}) is homogeneous and has constant coefficients. We now need to calculate the coefficients $A$ and $B$. For this purpose, we recall that \begin{equation}\mathbb{P}\left(U(t)+V(t)=ct,\; 0<U(t)<ct\right)=\frac{1}{2}\left(e^{\frac{\lambda}{2}t}-e^{-\lambda t}\right).\label{S1prob}\end{equation} We refer to the paper by Orsingher and Marchione \cite{orsinghermarchione} for details about formula (\ref{S1prob}). Hence, it must hold that $$\int_0^{ct}q(u,t)\,du=\frac{1}{2}\left(e^{\frac{\lambda}{2}t}-e^{-\lambda t}\right).$$
By using the relationship $$\int_0^t I_0\left(K\sqrt{s(t-s)}\right)\,ds=\frac{1}{K}\left(e^{\frac{Kt}{2}}-e^{-\frac{Kt}{2}}\right)$$
it can be verified that $$A=\frac{\lambda}{4c},\qquad\qquad B=\frac{1}{2c}$$
which completes the proof.
\end{proof}
Formula (\ref{q_distr}) describes the continuous distribution of the abscissa $U(t)$ when the process $(U(t),\,V(t))$ lies on $S^1_t$. The arguments above can be extended to the case of space-varying velocity. Hence, we now consider the distribution of $X(t)$ when the vector $(X(t),\,Y(t))$ lies on $R^1_t$. To study the probability density function 
$$h(x,t)=\mathbb{P}\left(\left(X(t),\,Y(t)\right)\in R^1_t,\;X(t)\in dx\right),\quad x\in\left\{w\in\mathbb{R}:\;0<\int_{x_0}^w\frac{dz}{\mathtt{v}(z)}<t\right\}$$ we consider, for $j=0,1$, the auxiliary density functions
\begin{align*}h_j(x,t)=\mathbb{P}\big(\left(X(t),\,Y(t)\right)\in R^1_t,\;X(t)\in& dx,\;D(t)=d_j\big),\nonumber\\& x\in\left\{w\in\mathbb{R}:\;0<\int_{x_0}^w\frac{dz}{\mathtt{v}(z)}<t\right\}.\end{align*}
By using the arguments discussed in section \ref{prelimsec}, it can be proved that
\begin{equation}\label{XYboundarysys}
\begin{dcases}
\frac{\partial h_0}{\partial t}=-\frac{\partial}{\partial x}\left[\mathtt{v}(x)\, h_0\right]+\frac{\lambda}{2}\,h_1-\lambda\,h_0\\
\frac{\partial h_1}{\partial t}=\frac{\lambda}{2}\,h_0-\lambda\,h_1.
\end{dcases}
\end{equation}
Clearly, the solution to the system (\ref{XYboundarysys}) is strictly related to the solution to (\ref{UVsys}). Indeed, the relationship (\ref{UVdef}) implies that, for $j=0,1$, \begin{equation}c\;q_j(u,t)=h_j(x,t)\, \mathtt{v}(x)\label{309}\end{equation} where $u=c\,\int_{x_0}^x\frac{dz}{\mathtt{v}(z)}$ and $q(u,t)$ is defined in formula (\ref{q_distr}). By taking into account that $h(x,t)=h_0(x,t)+h_1(x,t)$, formula (\ref{309}) permits us to obtain the exact distribution of $X(t)$ on $R^1_t$, which reads
\begin{align}h(x,t)=\frac{e^{-\lambda t}}{2\,\mathtt{v}(x)}\Bigg[\frac{\lambda}{2}&\,I_0\left(\lambda\;\sqrt{\int_{x_0}^x\frac{dz}{\mathtt{v}(z)}\,\left(t-\int_{x_0}^x\frac{dz}{\mathtt{v}(z)}\right)}\;\right)\nonumber\\&+\frac{\partial}{\partial t}\,I_0\left(\lambda\;\sqrt{\int_{x_0}^x\frac{dz}{\mathtt{v}(z)}\,\left(t-\int_{x_0}^x\frac{dz}{\mathtt{v}(z)}\right)}\;\right)\Bigg].\label{325}\end{align}

We conclude this section by deriving explicit formulas for the distribution of the process $(X(t),\,Y(t))$ under a specific choice of the velocity function. In particular, we consider the special case \begin{equation}\mathtt{v}(x)=c\,(1-x^2)\label{1-x^2}\end{equation} and starting point at the origin of the Cartesian plane. Observe that the chosen velocity vanishes at $x=\pm1$ with linear rate. As we will see, this implies that the time-varying support $R_t$ is bounded. Under the assumption (\ref{1-x^2}), the process $(X(t),\,Y(t))$ admits, in view of formula (\ref{UVdef}), the representation
\begin{equation}\label{UVdef(1-x^2)}\begin{cases}X(t)=\tanh\left(U(t)\right)\\Y(t)=\tanh\left(V(t)\right)\end{cases}\end{equation} where $(U(t),\,V(t))$ is the motion with orthogonal directions and constant velocity $c$. The relationship (\ref{UVdef(1-x^2)}) implies that the distribution of $(X(t),\,Y(t))$ in the interior of its support is given by
\begin{align*}f(x,y,t)=\frac{1}{(1-x^2)(1-y^2)}\;g\left(\frac{1}{2}\log\Bigg(\frac{1+x}{1-x}\right),\;\frac{1}{2}&\log\left(\frac{1+y}{1-y}\right),\;t\Bigg),\nonumber\\& (x,y)\in\text{Int}(R_t),\;t>0\end{align*}
where $g$ is defined in formula (\ref{pexactbiv}). We are now interested in obtaining the explicit representation of $R_t$ given the velocity (\ref{1-x^2}). For this purpose, we make use of the fact that the supports $S_t$ and $R_t$ are mapped onto each other by means of the change of variables (\ref{UVdef(1-x^2)}). Consider the top-right side of the square $S_t$, that is the set $S^1_t$. Clearly, this set is mapped into $$R^1_t=\left\{(x,y)\in\mathbb{R}^2: \;\;y=\frac{\tanh(ct)-x}{1-x\tanh(ct)},\, 0<x<\tanh(ct)\right\}.$$ By using a similar procedure for the other sides of $S_t$, the support $R_t$ of $(X(t),\,Y(t))$ can be obtained explicitly. A representation of the boundary $\partial R_t$ is shown in figure \ref{fig:Rt}, along with the equations describing its four sides.
\begin{figure}[h!]
\centering
\includegraphics[scale=0.7]{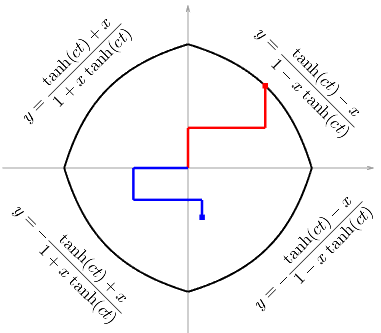}
\caption{boundary $\partial R_t$ of the support of the process $(X(t),\,Y(t))$ with space-varying velocity given by formula (\ref{1-x^2}). The equations describing the four sides of the boundary are given explicitly in the figure. It can be easily verified that, as $t\to+\infty$, the boundary approaches the square $[-1,1]^2$. The support can be obtained by applying a suitable nonlinear transformation to the square support of the process $(U(t),\,V(t))$ with constant velocity. Two sample paths are illustrated. The red path alternates only two contiguous directions, and thus corresponds to a particle lying on the boundary of the support.}
\label{fig:Rt}
\end{figure}
To study the distribution of $(X(t),\,Y(t))$ on $\partial R_t$, we consider, without loss of generality, the side $R^1_t$. The first way to describe the distribution of the process on $R^1_t$ is based on the use of the coordinate $$H(t)=\frac{1}{2}\log\left(\frac{(1+X(t))\,(1-Y(t))}{(1-X(t))\,(1+Y(t))}\right).$$
When $(X(t),\,Y(t))\in R^1_t$, it can be verified that $H(t)=0$ if and only if $(X(t),\,Y(t))$ lies at the midpoint of $R^1_t$. Similarly, the extrema of 
$R^1_t$ are identified by the condition $H(t)=\pm ct$. By using the coordinate $H(t)$, the distribution of the process on $R^1_t$ is immediately obtained by means of formula (\ref{Hdistribu}). Alternatively, the position of $(X(t),\,Y(t))$ on $\partial R_t$ can be identified by its abscissa as in formula (\ref{325}). In this case, it holds that
$$h(x,t)=\frac{1}{(1-x^2)}\;\,q\left(\frac{1}{2}\log\left(\frac{1+x}{1-x}\right),\,t\right).$$

\section{Telegraph process with direction-dependent velocity}
\noindent In this section, we consider a one-dimensional random motion $X(t)$ whose space-varying velocity depends on the direction $D(t)$. In particular, given two positive-valued continuously differentiable functions $\mathtt{v}_0(\cdot)$ and $\mathtt{v}_1(\cdot)$, we assume that the process moves with velocity $\mathtt{v}_0\left(X(t)\right)$ when it moves rightwards, that is when $D(t)=d_0$, and with velocity $-\mathtt{v}_1\left(X(t)\right)$ when $D(t)=d_1$. As usual, we assume that the direction changes are paced by a homogenous Poisson process $N(t)$ with intensity $\lambda$.\\
Our main conclusion is that, under the standard hydrodynamic limit where $\lambda$ and the velocities $\mathtt{v}_j(\cdot),\;j=0,1,$ tend to $\infty$ with a scaling such that $\frac{\lambda}{\mathtt{v}^2_j(\cdot)}$ remains finite, the process $X(t)$ does not exhibit a diffusive behavior. This conclusion is achieved by observing that, while diffusion processes have a mean which reflects the initial conditions of the process, the direction‐dependent speeds can induce, in general, a mean which is independent from the initial value $x_0$ of the process.\\
Denote by $f_j(x,t)$, for $j=0,1$, the density functions $$f_j(x,t)=\mathbb{P}\left(X(t)\in\mathop{dx},\,D(t)=d_j\right)/\mathop{dx}$$ and let $f(x,t)$ be the probability density function of $X(t)$, that is $$f(x,t)=f_0(x,t)+f_1(x,t).$$
It can be verified that the following system of partial differential equations holds:
\begin{equation}
\begin{dcases}
\frac{\partial f_0}{\partial t}=-\frac{\partial}{\partial x}\big[\mathtt{v}_0(x)\,f_0\big]+\lambda f_1-\lambda f_0\\
\frac{\partial f_1}{\partial t}=\frac{\partial}{\partial x}\big[\mathtt{v}_1(x)\,f_1\big]+\lambda f_0-\lambda f_1.
\end{dcases}\label{directiondependentsystem}
\end{equation}
However, in contrast to the cases previously examined in the present paper, we were not able to combine the equation in (\ref{directiondependentsystem}) to obtain a second-order partial differential equation for the density $f(x,t)$. Since such equation is the usual starting point for analyzing the hydrodynamic limit of finite-velocity processes, the inability to obtain such equation raises the question of whether the hydrodynamic limit of $X(t)$ admits a diffusive behaviour when the velocity is direction-dependent. To investigate this problem, we consider the special case
\begin{equation}\mathtt{v}_0(x)=c\,(1-x),\qquad \mathtt{v}_1(x)=c\,x\label{doublevelocity}\end{equation} with $c>0$, and we assume that $X(0)=x_0$, with $x_0\in(0,1)$. Clearly, in view of the functional forms of the velocity functions, the extrema of the interval $(0,1)$ act as unreachable barriers for $X(t)$. We are interested in studying the mean $\mathbb{E}\left[X(t)\right]$ of the process. We first introduce some preliminary notation. For integer $k\ge1$, we denote by $T_k$ the $k$-th arrival time of the Poisson process $N(t)$, that is $$T_k=\inf\{t:\;N(t)\ge k\}.$$ As usual, the paths of the Poisson process are assumed to be càdlàg. This implies that, for all $k\ge1$, $\forall t\in\left(T_{k-1},\,T_k\right),\; N(t)=N(T_{k-1}).$ Hence, in view of the assumption (\ref{doublevelocity}), it can be verified that
\begin{align}X(T_k)=&X(T_{k-1})\,e^{-c(T_k-T_{k-1})}\,\mathbbm{1}\{D(T_{k-1})=d_1\}\nonumber\\&\qquad\qquad+\left[1-\bigl(1-X(T_{k-1})e^{-c(T_k-T_{k-1})}\bigr)\right]\mathbbm{1}\{D(T_{k-1})=d_0\}\nonumber\\
=&X(T_{k-1})\,e^{-c(T_k-T_{k-1})}+\left(1-e^{-c(T_k-T_{k-1})}\right)\,\mathbbm{1}\{D(T_{k-1})=d_0\},\;\; k\ge1.\label{recursionTk}\end{align}
Observe that formula (\ref{recursionTk}) follows from the fact that, if $D(T_{k-1})=d_0$, then it holds that
\begin{equation}\frac{dX(t)}{dt}=c\,\left(1-X(t)\right),\qquad t\in(T_{k-1},\,T_k)\label{explanation1}\end{equation}
and similarly, if $D(T_{k-1})=d_1$, we have that
\begin{equation}\frac{dX(t)}{dt}=-c\,X(t),\qquad t\in(T_{k-1},\,T_k).\label{explanation2}\end{equation}
Hence, by integrating equations (\ref{explanation1}) and (\ref{explanation2}) over the interval $(T_{k-1},\,T_k)$, formula (\ref{recursionTk}) is obtained.\\
In general, for $T_{k-1}< t\le T_k$, it holds that
\begin{align*}X(t)=X(T_{k-1})\,e^{-c(t-T_{k-1})}+\left(1-e^{-c(t-T_{k-1})}\right)\,\mathbbm{1}\{D(T_{k-1})=d_0\},\qquad k\ge1.\end{align*}
In order to solve the recursion (\ref{recursionTk}), assume that, at the initial time $t=0$, the particle starts moving rightwards, that is $D(0)=d_0$. In this case, for $k\ge1$, it can be verified that
$$X(T_{2k})=X(T_{2k-1})\,e^{-c(T_{2k}-T_{2k-1})}$$ and $$X(T_{2k+1})=1-\bigl(1-X(T_{2k})\bigr)\,e^{-c(T_{2k+1}-T_{2k})}$$
which implies that 
\begin{equation}\label{reczerosol}X(T_{2k})=x_0e^{-c\,T_{2k}}+\sum_{j=0}^{2k-1}(-1)^{j+1}\,e^{-c(T_{2k}-T_j)}\end{equation}
and
\begin{equation}\label{reconesol}X(T_{2k+1})=1+x_0e^{-c\,T_{2k+1}}+\sum_{j=0}^{2k}(-1)^{j+1}\,e^{-c(T_{2k+1}-T_j)}.\end{equation}

\noindent In view of formulas (\ref{reczerosol}) and (\ref{reconesol}), it can be verified that, if $D(0)=d_0$, for $t\ge0$ and $k\ge 0$ it holds that 
\begin{equation}\label{solrec1}X(t)=\begin{dcases}1+x_0\,e^{-ct}+\sum_{j=0}^{2k}(-1)^{j+1}e^{-c(t-T_j)}\qquad &\text{if}\;T_{2k}<t\le T_{2k+1}\\
x_0\,e^{-ct}+\sum_{j=0}^{2k+1}(-1)^{j+1}e^{-c(t-T_j)}\qquad &\text{if}\;T_{2k+1}<t\le T_{2k+2}\end{dcases}\end{equation}

\noindent Similarly, if $D(0)=d_1$ it holds that
$$X(T_{2k})=1-\bigl(1-X(T_{2k-1})\bigr)\,e^{-c(T_{2k}-T_{2k-1})}$$ and $$X(T_{2k+1})=X(T_{2k})\,e^{-c(T_{2k+1}-T_{2k})}$$
which implies that 
$$X(T_{2k})=\sum_{j=0}^{2k-1}(-1)^j\,e^{-c(T_{2k}-T_j)}-(1-x_0)e^{-c\,T_{2k}}+1$$
and
$$X(T_{2k+1})=\sum_{j=0}^{2k}(-1)^j\,e^{-c(T_{2k+1}-T_j)}-(1-x_0)e^{-c\,T_{2k+1}}.$$
Hence, if $D(0)=d_1$, for $t\ge0$ and $k\ge 0$ it holds that 
\begin{equation}\label{solrec2}X(t)=\begin{dcases}-(1-x_0)\,e^{-ct}+\sum_{j=0}^{2k}(-1)^{j}e^{-c(t-T_j)}\qquad &\text{if}\;T_{2k}<t\le T_{2k+1}\\
1-(1-x_0)\,e^{-ct}+\sum_{j=0}^{2k+1}(-1)^{j}e^{-c(t-T_j)}\qquad &\text{if}\;T_{2k+1}<t\le T_{2k+2}\end{dcases}\end{equation}

\noindent We now recall that, conditional on the event $N(t)=n$, the arrival time $T_j$ follows the distribution of the $j$-th order statistic of $n$ independent and identically distributed uniform variables over the interval $(0,t)$. In other words, it holds that $$\mathbb{P}\left(T_j\in\mathop{ds}\big\lvert N(t)=n\right)/\mathop{ds}=\frac{n!}{(j-1)!\,(n-j)!}\,\frac{s^{j-1}\,(t-s)^{n-j}}{t^n},\qquad s\in(0,t).$$
This implies that \begin{align}\mathbb{E}\left[e^{c\,T_j}\big\lvert N(t)=n\right]=&\frac{n!}{(j-1)!\,(n-j)!}\;\frac{1}{t^n}\int_0^t e^{cs} s^{j-1}(t-s)^{n-j}ds\nonumber\\
=&\frac{n!}{(j-1)!\,(n-j)!}\int_0^1 e^{ctu} u^{j-1}(1-u)^{n-j}du.\label{betapois}\end{align}
By using the integral representation of the confluent hypergeometric function
$${}_1F_1(a;b;z)=\frac{\Gamma(b)}{\Gamma(a)\,\Gamma(b-a)}\,\int_0^1e^{zu}u^{a-1}(1-u)^{b-a-1}du,\qquad z\in\mathbb{C}$$
with $\Re(a)>0,\;\Re(b)>0$, we can express formula (\ref{betapois}) as
\begin{equation}\label{Tjhyp}\mathbb{E}\left[e^{c\,T_j}\big\lvert N(t)=n\right]={}_1F_1\big(j;n+1;ct\big).\end{equation}
By taking the conditional expectations of (\ref{solrec1}) and (\ref{solrec2}) using formula (\ref{Tjhyp}), and recalling that formulas (\ref{solrec1}) and (\ref{solrec2}) respectively hold if $D(0)=d_0$ and $D(0)=d_1$, we obtain the following conditional expected values for $k\ge0$:
\begin{flalign*}
\mathbb{E}\left[X(t)\big\lvert D(0)=d_0,\, N(t)=2k\right]=1+x_0 e^{-ct}-e^{-ct}\sum_{j=0}^{2k}(-1)^j{}_1F_1\big(j;2k+1;ct\big)
\end{flalign*}
\vspace{-12mm}

\begin{flalign*}
\mathbb{E}\left[X(t)\big\lvert D(0)=d_0,\, N(t)=2k+1\right]=x_0e^{-ct}-e^{-ct}\sum_{j=0}^{2k+1}(-1)^j{}_1F_1\big(j;2(k+1);ct\big)
\end{flalign*}
\vspace{-12mm}

\begin{flalign*}
\mathbb{E}\left[X(t)\big\lvert D(0)=d_1,\, N(t)=2k\right]=-(1-x_0)e^{-ct}+e^{-ct}\sum_{j=0}^{2k}(-1)^j{}_1F_1\big(j;2k+1;ct\big)
\end{flalign*}
\vspace{-12mm}

\begin{flalign*}
\mathbb{E}&\left[X(t)\big\lvert D(0)=d_1,\, N(t)=2k+1\right]\\&\qquad\qquad=1-(1-x_0)e^{-ct}+e^{-ct}\sum_{j=0}^{2k+1}(-1)^j{}_1F_1\big(j;2(k+1);ct\big).\qquad\qquad\qquad\end{flalign*}
By using these formulas, it is straightforward to verify that the mean of the process $X(t)$ is independent of the number of direction changes that have occurred up to time $t$. Indeed, for all $n\ge0$, it holds that \begin{equation}\label{conditionaln}\mathbb{E}\left[X(t)\big\lvert N(t)=n\right]=x_0e^{-ct}+\frac{1-e^{-ct}}{2}.\end{equation}
Clearly, formula (\ref{conditionaln}) coincides with the unconditional mean too, that is
\begin{equation}\label{unconditionaln}\mathbb{E}\left[X(t)\right]=x_0e^{-ct}+\frac{1-e^{-ct}}{2}.\end{equation}
We now take the hydrodynamic limit for $\lambda,c\to+\infty$ of formula (\ref{unconditionaln}), which yields
\begin{equation}\lim_{\lambda,c\to+\infty}\mathbb{E}\left[X(t)\right]=\frac{1}{2}.\end{equation}
Hence, in the hydrodynamic limit, the mean of the process converges to the point $x=\frac{1}{2}$, losing any dependence on the initial position $x_0$. This behaviour is not consistent with that of a diffusion process, for which the distribution of the process should depend on the initial condition and spread around the starting point over time. We also emphasize that the asymptotic mean is independent of $t$. Hence, even if the limiting process starts at a point $x_0\neq\frac{1}{2}$, its mean approaches $\frac{1}{2}$ in an arbitrarily small time, possibly indicating some degenerate behaviour of $X(t)$ in the hydrodynamic limit. We leave the investigation of the asymptotic behavior of the process to future work. However, we conducted some numerical simulations of the process for large values of $c$, with $\lambda=c^2$. The results suggest that approximately half of the sample paths tend to rapidly collapse to the point $x=0$, while the other half collapses near $x=1$ in a small time lapse $t$.

\section{Random motions with time-dependent velocity}
\noindent In this section, we study the behaviour of the telegraph process $\{X(t)\}_{t\ge0}$ under the assumption that the velocity is time-dependent. In particular, we assume that the velocity of the process at time $t$ is $$\mathtt{v}(t)=c\,\sigma(t)$$ where $c$ is a stricly positive constant and $\sigma:[0,+\infty)\to\mathbb{R}^+$ is a positive-valued function of time. Moreover, we impose an additional technical condition on $\sigma(\cdot)$, whose relevance will be clear later. In particular, we assume that $\sigma(\cdot)$ is square-integrable, in the sense that, for all $t>0$, $\int_0^t\sigma^2(s)\,ds<+\infty$. By adopting the usual notation for the probability density functions $$f_j(x,t)=\mathbb{P}\big(X(t)\in dx,\;D(t)=d_j\big),\qquad j=0,1$$ and assuming that $X(0)=0$, it is clear that the system of partial differential equations
\begin{equation}
\begin{dcases}
\frac{\partial f_0}{\partial t}=-c\, \sigma(t)\frac{\partial f_0}{\partial x}+\lambda f_1-\lambda f_0\\
\frac{\partial f_1}{\partial t}=c\, \sigma(t)\frac{\partial f_1}{\partial x}+\lambda f_0-\lambda f_1
\end{dcases}\label{sigmatsys}
\end{equation}
holds with initial conditions $f_0(x,0)=f_1(x,0)=\frac{1}{2}\,\delta(x)$. Hence, the probability density function $f$ defined as $$f(x,t)=f_0(x,t)+f_1(x,t)$$ is a solution to the partial differential equation
\begin{equation}\frac{\partial}{\partial t}\left[\frac{1}{\sigma(t)}\,\frac{\partial f}{\partial t}\right]+\frac{2\lambda}{\sigma(t)}\,\frac{\partial f}{\partial t}=c^2\sigma(t)\,\frac{\partial^2 f}{\partial x^2}.\label{sigmapde}\end{equation}
Stochastic processes whose distribution is governed by equation (\ref{sigmapde}) have been investigated by Masoliver and Weiss \cite{masoliverweiss}. As the authors pointed out, finding an explicit solution to equation (\ref{sigmapde}) is a difficult task. By using a Wentzel–Kramers–Brillouin approximation, they were able to obtain an approximate solution for large values of $\lambda$, which concentrates the probability mass near the starting point.\\
We now discuss the hydrodynamic limit of the process $X(t)$ when both $\lambda$ and $c$ tend to $+\infty$, with $\frac{\lambda}{c^2}\to1$. Dividing equation (\ref{sigmapde}) by $c^2$ and taking the limit for $\lambda,c\to+\infty$, we obtain the diffusion-type equation
\begin{equation}\,\frac{\partial f}{\partial t}=\frac{\sigma^2(t)}{2}\,\frac{\partial^2 f}{\partial x^2}.\label{sigmapdelimit}\end{equation}
Solving the partial differential equation (\ref{sigmapdelimit}) yields that, in the hydrodynamic limit, the process $X(t)$ follows the distribution
\begin{equation}\lim_{\lambda,c\to+\infty}f(x,t)=\frac{1}{\sqrt{2\pi \int_0^t\sigma^2(u)\mathop{du}}}\;\exp\left(-\frac{x^2}{2\int_0^t\sigma^2(u)\mathop{du}}\right).\label{sigmatdistr}\end{equation} The key for interpreting the limiting distribution (\ref{sigmatdistr}) is the following. We first claim that, before passing to the hydrodynamic limit, the process $X(t)$ emerging from the system (\ref{sigmatsys}) admits the explicit representation \begin{equation}X(t)=\int_0^t\sigma(u)\,d\mathcal{T}(u)\label{intdT}\end{equation} where $\{\mathcal{T}(t)\}_{t\ge0}$ is a standard telegraph process with velocity $c$. This can be proved by applying the methods described in section \ref{prelimsec}. We now observe that  the right-hand side of equation (\ref{sigmatdistr}) coincides with the distribution of the It\^o integral $\int_0^t\sigma(u)\,dB(u)$, where $\{B(t)\}_{t\ge0}$ is a standard Brownian motion. Hence, the relationship (\ref{sigmatdistr}) indicates that the following identity in distribution holds:
\begin{equation}\lim_{\lambda,c\to+\infty} X(t)\overset{i.d.}{=}\int_0^t\sigma(u)\,dB(u).\label{itolimit}\end{equation} In view of the representation (\ref{intdT}) for the process $X(t)$, the identity in distribution (\ref{itolimit}) is consistent with the well-known fact the the telegraph process $\mathcal{T}(t)$ converges, in the hydrodynamic limit, to a standard Brownian motion.\\
\noindent While obtaining the exact distribution of the process $X(t)$ is an open problem, it is interesting to observe that its covariance structure can be exactly determined.
\begin{thm}For $s,t\ge0$, the autocovariance of the process $X(t)$ is
\begin{equation}\label{autocovbeforelimit}\mathbb{E}\left[X(s)\,X(t)\right]=c^2\int_0^t\int_0^s e^{-2\lambda\lvert x-y\lvert}\sigma(x)\sigma(y)\mathop{dx}\mathop{dy}\end{equation} provided that the integral to the right-hand side of the equation exists.\end{thm}
\begin{proof}To prove the theorem, we use the following representation for the standard telegraph process:
\begin{equation}\label{telrepr}\mathcal{T}(t)=V(0)\int_0^t(-1)^{N(u)}\,du\end{equation}
where $V(0)$ is a random variable defined as $$V(0)=\begin{cases}c\qquad &\text{with prob. }\frac{1}{2}\\-c\qquad &\text{with prob. }\frac{1}{2}\end{cases}$$ and independent from $\{N(t)\}_{t\ge0}$. By substituting formula (\ref{telrepr}) into (\ref{intdT}), we have that
\begin{equation*}X(t)=V(0)\int_0^t\sigma(u)\,(-1)^{N(u)}\,du\end{equation*} Thus, it holds that 
\begin{equation}\mathbb{E}\left[X(s)\,X(t)\right]=c^2\int_0^t\int_0^s \mathbb{E}\left[(-1)^{N(x)+N(y)}\right]\,\sigma(x)\sigma(y)\mathop{dx}\mathop{dy}.\label{hydroaut}\end{equation}
The proof is completed by observing that
\begin{equation*}\mathbb{E}\left[(-1)^{N(x)+N(y)}\right]=e^{-2\lambda\lvert x-y\lvert}.\end{equation*}
\end{proof}
It is interesting to observe that formula (\ref{autocovbeforelimit}) is consistent with the hydrodynamic limit given in formulas (\ref{sigmatdistr}) and (\ref{itolimit}). In particular, for $\lambda,c\to+\infty$ with the usual scaling, we have that $$\lim_{\lambda,c\to+\infty}\mathbb{E}\left[X(s)\,X(t)\right]=\int_0^{t\wedge s}\sigma^2(u)\,du$$ which coincides with the autocovariance function of the It\^o integral (\ref{itolimit}). A heuristic explanation of this fact can be given by observing that, for large values of $\lambda,c$ with $\frac{\lambda}{c^2}\simeq1$,  the kernel function $c^2\,e^{-2\lambda\lvert x-y\lvert}$ tends to concentrate the probability mass near the set $\{(x,y)\in\mathbb{R}^2:\;x=y\}.$ In other words, $$c^2\,e^{-2\lambda\lvert x-y\lvert}\simeq \delta(\lvert y-x\lvert)$$ which implies, in view of formula (\ref{hydroaut}), that 
\begin{equation*}\mathbb{E}\left[X(s)\,X(t)\right]\simeq\int_0^t\int_0^s\delta(\lvert x-y\lvert)\,\sigma(x)\sigma(y)\mathop{dx}\mathop{dy}=\int_0^{t\wedge s}\sigma^2(u)\,du.\end{equation*} Of course, a rigorous proof can be given under suitable assumptions on $\sigma(\cdot)$. In the following theorem, we present a proof assuming that $\sigma(\cdot)$ is continuous and finite-valued.
\begin{thm}Let $\sigma:[0,+\infty)$ be a strictly-positive continuous function such that $\sigma(t)<+\infty$ for all $t\in\mathbb{R}^+$. Then, for $\lambda,c\to+\infty$ such that $\frac{\lambda}{c^2}\to1$, it holds that 
$$\lim_{\lambda,c\to+\infty}\mathbb{E}\left[X(s)\,X(t)\right]=\int_0^{t\wedge s}\sigma^2(u)\,du,\qquad\forall s,t\ge0.$$\end{thm}
\begin{proof}We must verify that 
$$\lim_{\lambda,c\to+\infty}c^2\int_0^t\int_0^s e^{-2\lambda\lvert x-y\lvert}\sigma(x)\sigma(y)\mathop{dx}\mathop{dy}=\int_0^{t\wedge s}\sigma^2(u)\,du,\qquad\forall s,t\ge0.$$ For simplicity, we consider the case $s=t$. The extension to the case $s\neq t$ is straightforward. We divide the integration set $[0,t]^2$ into five subsets as illustrated in figure \ref{fig:strip}.
\begin{figure}[h!]
\centering
\includegraphics[scale=0.68]{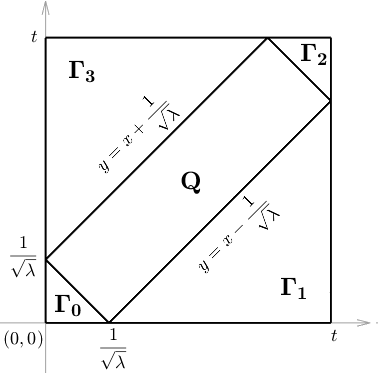}
\caption{subdivision of the integration set $[0,t]^2$ used to prove the hydrodynamic limit of the covariance function of the process $X(t)$. Observe that, as $\lambda\to+\infty$, the strip $Q$ collapses onto the diagonal of the square.}
\label{fig:strip}
\end{figure}
Thus, we can write that
\begin{align}c^2\int_0^t\int_0^t e^{-2\lambda\lvert x-y\lvert}\sigma(x)\sigma(y)\mathop{dx}\mathop{dy}=c^2\iint_{\substack{Q}}&e^{-2\lambda\lvert x-y\lvert}\sigma(x)\sigma(y)\mathop{dx}\mathop{dy}\nonumber\\&+c^2\sum_{j=0}^3\iint_{\substack{\Gamma_j}} e^{-2\lambda\lvert x-y\lvert}\sigma(x)\sigma(y)\mathop{dx}\mathop{dy}.\label{zonti}\end{align}
We claim that, in the formula above, the integrals over the sets $\Gamma_j,\;j=0,1,2,3,$ converge to 0 as $\lambda,c\to+\infty$. We verify the claim for the integral over $\Gamma_1$, as the underlying idea can be easily applied to the other sets too. By defining $k_t=\max_{s\in[0,t]}\sigma(s)$, and observing that $k_t<+\infty$ in view of the assumptions on $\sigma(\cdot)$, we clearly have that 
\begin{align}0\le c^2\iint_{\substack{\Gamma_1}} e^{-2\lambda\lvert x-y\lvert}\sigma(x)\sigma(y)\mathop{dx}\mathop{dy}&\le c^2 k_t^2\iint_{\substack{\Gamma_1}} e^{-2\lambda\lvert x-y\lvert}\mathop{dx}\mathop{dy}\nonumber\\&=\frac{c^2k_t^2}{4\lambda^2}\left[\left(2\lambda t-2\sqrt{\lambda}-1\right)e^{-2\sqrt{\lambda}}+e^{-2\lambda t}\right].\label{to0}\end{align}
Since the right-hand side of formula (\ref{to0}) converges to 0 as $\lambda,c\to+\infty$, the integral over $\Gamma_1$ must converge to 0 too. The same result holds for all four sets $\Gamma_j,\,j=0,1,2,3$. Hence, formula (\ref{zonti}) implies that
\begin{equation}\lim_{\lambda,c\to+\infty}c^2\int_0^t\int_0^t e^{-2\lambda\lvert x-y\lvert}\sigma(x)\sigma(y)\mathop{dx}\mathop{dy}=\lim_{\lambda,c\to+\infty}c^2\iint_{\substack{Q}}e^{-2\lambda\lvert x-y\lvert}\sigma(x)\sigma(y)\mathop{dx}\mathop{dy}.\label{zonti2}\end{equation} To compute the limit, we reformulate the right-hand side of equation (\ref{zonti2}) by observing that, in view of the change of variables $$\begin{cases}u=\frac{x+y}{2}\\v=\frac{x-y}{2}\end{cases}$$ it holds that
\begin{equation}c^2\iint_{\substack{Q}}e^{-2\lambda\lvert x-y\lvert}\sigma(x)\sigma(y)\mathop{dx}\mathop{dy}=2c^2\int_{\frac{1}{2\sqrt{\lambda}}}^{t-\frac{1}{2\sqrt{\lambda}}}\int_{-\frac{1}{2\sqrt{\lambda}}}^{\frac{1}{2\sqrt{\lambda}}} e^{-4\lambda\lvert v\lvert}\sigma\left(u+v\right)\sigma\left(u-v\right)\mathop{dv}\mathop{du}.\label{zonti99}\end{equation} By applying the mean value theorem for integrals, there exists a real number $\varepsilon(\lambda)\in[-\frac{1}{2\sqrt{\lambda}},\frac{1}{2\sqrt{\lambda}}]$ such that \begin{align}\int_{-\frac{1}{2\sqrt{\lambda}}}^{\frac{1}{2\sqrt{\lambda}}} e^{-4\lambda\lvert v\lvert}\sigma\left(u+v\right)\sigma\left(u-v\right)\mathop{dv}=&\sigma\big(u+\varepsilon(\lambda)\big)\,\sigma\big(u-\varepsilon(\lambda)\big)\int_{-\frac{1}{2\sqrt{\lambda}}}^{\frac{1}{2\sqrt{\lambda}}} e^{-4\lambda\lvert v\lvert}\mathop{dv}\nonumber\\=&\sigma\big(u+\varepsilon(\lambda)\big)\,\sigma\big(u-\varepsilon(\lambda)\big)\,\frac{1-e^{-2\sqrt{\lambda}}}{2\lambda}.\label{zonti100}\end{align} By substituting formula (\ref{zonti100}) into (\ref{zonti99}) we obtain
\begin{equation}c^2\iint_{\substack{Q}}e^{-2\lambda\lvert x-y\lvert}\sigma(x)\sigma(y)\mathop{dx}\mathop{dy}=\frac{c^2}{\lambda}\left(1-e^{-2\sqrt{\lambda}}\right)\int_{\frac{1}{2\sqrt{\lambda}}}^{t-\frac{1}{2\sqrt{\lambda}}}\sigma\big(u+\varepsilon(\lambda)\big)\,\sigma\big(u-\varepsilon(\lambda)\big)\mathop{du}.\label{zonti101}\end{equation} Hence, by combining formulas (\ref{zonti2}) and (\ref{zonti101}) and observing that $\lim_{\lambda\to+\infty}\varepsilon(\lambda)=0$, the proof is complete.
\end{proof}

\section{Financial applications}
Although the core of the present paper is primarily theoretical, this section explores a potential financial application of our results, focusing on multivariate motions with space-varying velocities. Financial applications of univariate finite-velocity random motions have been studied by several authors, mainly in the context of financial derivatives modeling and pricing (see, for instance, Di Crescenzo and Pellerey \cite{dicrescenzo}, Ratanov and Melnikov \cite{ratanovmelnikov}, López and Ratanov \cite{lopezratanov} and references therein). The key underlying idea is the following. Let $\{\mathcal{T}(t)\}_{t\ge0}$ be a standard telegraph process with constant velocity $c>0$ and direction changes paced by a homogeneous Poisson process with intensity $\lambda$. In principle, the process $\mathcal{T}(t)$ can be employed as a finite-velocity counterpart of the standard Brownian motion, and a finite-velocity geometric Brownian motion can be defined accordingly. In particular, as mentioned in section \ref{prelimsec}, one can define the process \begin{equation}X(t)=x_0\,e^{\mathcal{T}(t)}\label{FVBS}\end{equation} with $x_0>0$. Under the hydrodynamic limit for $\lambda,c\to+\infty$ with $\frac{\lambda}{c^2}\to1$, the process $X(t)$ converges in distribution to a geometric Brownian motion, that is $$\lim_{\lambda,c\to+\infty}X(t)\overset{i.d.}{=}x_0\,e^{B(t)}$$ where $\{B(t)\}_{t\ge0}$ represents a standard Brownian motion. The geometric Brownian motion is widely used in practice to model stock prices in financial markets. Although empirical evidence indicates that log-returns often exhibit heavier tails than the Gaussian distribution, the lognormal model for stock prices provides a reasonable approximation while maintaining high analytical tractability. For this reason, ever since Black and Scholes \cite{blackscholes} published their seminal article, the geometric Brownian motion has consistently gained considerable popularity among both practitioners and academics for financial modeling. The process (\ref{FVBS}) provides a finite-velocity alternative to the Black-Scholes model. Under the assumption that the dynamics of the stock's price is driven by a telegraph process, the problem of pricing European options was investigated by López and Ratanov \cite{lopezratanov}. We emphasize that a compound-Poisson jump component is usually included in the process (\ref{FVBS}) in order to rule out arbitrage opportunities (we refer to Di Crescenzo and Pellerey \cite{dicrescenzo} and Ratanov and Melnikov \cite{ratanovmelnikov} for details). We omit this component for the sake of simplicity, since its inclusion does not affect the validity of the arguments that follow.\\
\noindent When modeling financial instruments, univariate models are commonly employed for pricing and hedging purposes. However, in risk management, these models are typically insufficient, as they fail to capture the dependencies among financial variables that are crucial for assessing portfolio risk. For instance, in portfolio optimization problems, the correlation between the prices of different market stocks must be taken into account, thus requiring the adoption of suitable multivariate models. For mathematical discussions on the role of correlation in portfolio composition and risk management, we refer to the classical work by Markowitz \cite{Markowitz}. Restricting the analysis to the bivariate case for simplicity, the correlation between two stocks in the Black-Scholes framework can be modeled by considering two correlated standard Brownian motions $\{B_1(t)\}_{t\ge0}$ and $\{B_2(t)\}_{t\ge0}$ such that $\mathbb{E}\left[B_1(t)\,B_2(t)\right]=\rho t$ for all $t>0$, with $\rho\in(-1,1)$. Thus, a bivariate extension of the Black-Scholes model can be obtained by defining the bivariate process $\big(X(t),Y(t)\big)$ as
\begin{equation}\label{bivariateGBM}\begin{cases}X(t)=x_0\,e^{\left(\mu-\frac{\sigma^2}{2}\right) t + \sigma B_1(t)}\\Y(t)=y_0\,e^{\left(\kappa-\frac{\eta^2}{2}\right) t + \eta B_2(t)}\end{cases}\end{equation} where $\mu,\kappa\in\mathbb{R}$ and $x_0,y_0,\sigma,\eta>0$. Equivalently, the process (\ref{bivariateGBM}) can be represented as an It\^{o} diffusion by means of the stochastic differential equations
\begin{equation}\label{bivariateGBMITO}\begin{cases}dX(t)=\mu X(t)\mathop{dt}+\sigma X(t)\,dB_1(t)\\dY(t)=\kappa Y(t)\mathop{dt}+\eta Y(t)\,dB_2(t)\end{cases}\end{equation} subject to the initial conditions $X(0)=x_0$ and $Y(0)=y_0$. By interpreting the processes $X(t)$ and $Y(t)$ as the prices of two distinct stocks at time $t$, formula (\ref{bivariateGBMITO}) clearly describes the dynamics of two correlated stocks, each modeled as a geometric Brownian motion. The dependence between the stock prices is captured by the correlation coefficient $\rho$ between the Brownian motions $B_1(t)$ and $B_2(t)$.\\
Within the framework of financial modeling with telegraph processes, one may be interested in constructing a finite-velocity analogue of the bivariate model (\ref{bivariateGBMITO}). This involves constructing a finite-velocity bivariate process $\big(X(t),Y(t)\big)$ which converges, in the hydrodynamic limit, to the bivariate correlated geometric Brownian motion (\ref{bivariateGBMITO}). The aim of the present section of the paper is to illustrate how this construction can be achieved by extending the ideas described in section \ref{sec:multivariate}.\\
Consider a planar finite-velocity random motion $\big((X(t),Y(t)\big)$ with orthogonal directions $d_j,\;j=0,1,2,3$. The intensity of the of the Poisson process which governs the direction changes is denoted by $\lambda$, and the space-varying velocity of the process $\big((X(t),Y(t)\big)$ is assumed to be $$\mathtt{v}(x)=c\,x$$ with $c>0$. In section \ref{sec:multivariate}, we assumed that the process turns clockwise or counterclockwise with equal probability $\frac{1}{2}$ at each Poisson event. We now adopt a slight generalization of this rule, which was first introduced in the paper by Marchione and Orsingher \cite{marchione}. In particular, we assume that the direction $D(t)$ of $\big((X(t),Y(t)\big)$ is a continuous-time Markov process with state space $\{d_0,d_1,d_2,d_3\}$ and  generator matrix
\begin{equation}\label{Generator}G=\begin{pmatrix}-\lambda  & \lambda p  & 0 & \lambda (1-p)  \\ \lambda p  & - \lambda  &  \lambda (1-p)  & 0 \\ 0 & \lambda (1-p) & -\lambda & \lambda p \\ \lambda(1-p) &0 &\lambda p & -\lambda\end{pmatrix}\end{equation} with $p\in[0,1]$. An intuitive description of the rules which govern the direction changes is given in figure \ref{fig:directionchanges}.\\
\begin{figure}[h]
\centering
\includegraphics[scale=0.65]{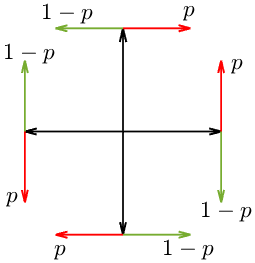}
\caption{the black arrows indicate the four possible directions of motion for the process $\big(X(t), Y(t)\big)$. For each direction, the colored arrows show the possible transitions when a direction change occurs: red arrows correspond to directions changes that occur with probability $p$, while green arrows correspond to direction changes that occur with probability $1-p$.}
\label{fig:directionchanges}
\end{figure}
\noindent We denote the bounded support of  $\big(X(t), Y(t)\big)$ as $R_t$. Consider the probability density function of $\big(X(t), Y(t)\big)$ in the interior of $R_t$, that is
$$f(x,y,t)\mathop{dx}\mathop{dy}=\mathbb{P}\big(X(t)\in\mathop{dx},\,Y(t)\in\mathop{dy}\big),\qquad (x,y)\in \text{Int}(R_{t}),\,t>0.$$
and, for $j=0,1,2,3$, define the auxiliary density functions
$$f_j(x,y,t)\mathop{dx}\mathop{dy}=\mathbb{P}\big(X(t)\in\mathop{dx},\,Y(t)\in\mathop{dy},\,D(t)=d_j\big),\;(x,y)\in \text{Int}(R_{t}),\,t>0.$$  Clearly, $f(x,y,t)=\sum_{j=0}^3 f_j(x,y,t).$ In view of the discussions presented in the previous sections, it can be verified that the following system of partial differential equations holds:
\begin{equation}
\begin{dcases}
\frac{\partial f_0}{\partial t}=-c\frac{\partial}{\partial x}\big[x\,f_0\big]+\lambda p f_1+\lambda(1-p) f_3-\lambda f_0\\
\frac{\partial f_1}{\partial t}=-c\frac{\partial}{\partial y}\big[y\,f_1\big]+\lambda p f_0+\lambda(1-p) f_2-\lambda f_1\\
\frac{\partial f_2}{\partial t}=c\frac{\partial}{\partial x}\big[x\,f_2\big]+\lambda(1-p) f_1+\lambda p f_3-\lambda f_2\\
\frac{\partial f_3}{\partial t}=c\frac{\partial}{\partial y}\big[y\,f_3\big]+\lambda(1-p) f_0+\lambda p f_2-\lambda f_3.
\end{dcases}\label{geomplanar}
\end{equation}
In order to find the exact distribution of $\big(X(t), Y(t)\big)$, we first use the transformation (\ref{UVdef}) to define the process $\big(U(t), V(t)\big)$ as
\begin{equation}\begin{cases}U(t)=\log\left(\frac{X(t)}{x_0}\right)\\V(t)=\log\left(\frac{Y(t)}{y_0}\right).\end{cases}\label{tardi}\end{equation} Clearly, the process $\big(U(t), V(t)\big)$ has constant velocity $c$, and its support coincides with the square $S_{t}$ defined in formula (\ref{supportgenerali}). Denoting by $g_j(x,t,t),\;j=0,1,2,3$, the density functions 
$$g_j(x,y,t)\mathop{du}\mathop{dv}=\mathbb{P}\big(U(t)\in\mathop{du},\,V(t)\in\mathop{dv},\,D(t)=d_j\big),\; (u,v)\in \text{Int}(S_{t}),\,t>0$$
formulas (\ref{geomplanar}) and (\ref{tardi}) imply that
\begin{equation}
\begin{dcases}
\frac{\partial g_0}{\partial t}=-c\frac{\partial g_0}{\partial u}+\lambda p\,g_1+\lambda(1-p) g_3-\lambda g_0\\
\frac{\partial g_1}{\partial t}=-c\frac{\partial g_1}{\partial v}+\lambda p\,g_0+\lambda(1-p) g_2-\lambda g_1\\
\frac{\partial g_2}{\partial t}=c\frac{\partial g_2}{\partial u}+\lambda(1-p) g_1+\lambda p\,g_3-\lambda g_2\\
\frac{\partial g_3}{\partial t}=c\frac{\partial g_3}{\partial v}+\lambda(1-p) g_0+\lambda p\,g_2-\lambda g_3.
\end{dcases}\label{geomplanar2}
\end{equation}
The system (\ref{geomplanar2}) implies that the process $\big(U(t), V(t)\big)$ has constant velocity and switching rule given by formula (\ref{Generator}). This process was studied by Marchione and Orsingher \cite{marchione}, who proved that its probability density function $g(u,v,t)$
satisfies the fourth-order partial differential equation
\begin{align}\Bigg\{\left(\frac{\partial}{\partial t}+\lambda\right)^4-\bigg(c^2&\Delta+2\lambda^2\Big(p^2+(1-p)^2\Big)\bigg)\left(\frac{\partial}{\partial t}+\lambda\right)^2+c^4\frac{\partial^4}{\partial u^2\partial v^2}\nonumber\\
&+2\lambda^2c^2(1-2p)\frac{\partial^2}{\partial u\,\partial v}+\lambda^4(1-2p)^2\Bigg\}g=0\label{fourthorderpde}\end{align}
\noindent where $\Delta=\frac{\partial^2}{\partial u^2}+\frac{\partial^2}{\partial v^2}$ represents the bivariate Laplacian. The authors also investigated the hydrodynamic limit of equation (\ref{fourthorderpde}). In particular, they proved that, for $\lambda,c\to+\infty$ with $\frac{\lambda}{c^2}\to1$, the following limiting partial differential equation is obtained
\begin{equation}\frac{\partial g}{\partial t}=\frac{\Delta g}{8p(1-p)}+\frac{2p-1}{4p(1-p)}\,\frac{\partial^2 g}{\partial u\,\partial v}.\label{hydropde}\end{equation}
Thus, by combining equation (\ref{hydropde}) with formula (\ref{tardi}), it can be verified that the probability density function $f(x,y,t)$ of $\big(X(t), Y(t)\big)$ satisfies, in the hydrodynamic limit, the partial differential equation
\begin{align}\label{kolmogorovequation}
\frac{\partial f}{\partial t}=&\frac{f}{4p(1-p)}+\frac{3}{8p(1-p)}\left(x\frac{\partial f}{\partial x}+y\frac{\partial f}{\partial y}\right)\nonumber\\
&\;+\frac{1}{8p(1-p)}\left(x^2\frac{\partial^2 f}{\partial x^2}+y^2\frac{\partial^2 f}{\partial y^2}\right)+\frac{2p-1}{4p(1-p)}\,\frac{\partial^2}{\partial x\,\partial y}\big(xy\,f\big).\end{align}
It is now a matter of straightforward calculation to verify that equation (\ref{kolmogorovequation}) coincides with the Kolmogorov forward equation of the diffusion process (\ref{bivariateGBMITO}) with parameters \begin{equation}\label{paramito}\mu=\kappa=\frac{1}{8p(1-p)},\qquad \sigma^2=\eta^2=\frac{1}{4p(1-p)},\qquad \rho=2p-1.\end{equation}
In other words, the process $\big(X(t), Y(t)\big)$ represents the desired finite-velocity counterpart of the bivariate correlated geometric Brownian motion (\ref{bivariateGBMITO}). Thus, within the framework of financial modeling with finite-velocity random motions, this process provides a description of the dynamics of two dependent stocks, where the parameter $p$ controls the dependence between the marginals. Observe that the parametrization (\ref{paramito}) entails no loss of generality, since any more general parametrization can be achieved by means of trivial transformations of the components of the random vector $\big(X(t), Y(t)\big)$. Finally, we remark that the exact distribution of $\big(X(t), Y(t)\big)$ can be obtained. Indeed, it was proved by Marchione and Orsingher \cite{marchione} that
\begin{align}\mathbb{P}\Big(&U(t)\in du,\,V(t)\in dv\Big)\nonumber\\=&\frac{e^{-\lambda t}}{2c^2}\bigg[\lambda(1-p) \, I_0\bigg(\frac{\lambda(1-p)}{c}\sqrt{c^2t^2-(u+v)^2}\bigg)\nonumber\\&\qquad\qquad+\frac{\partial}{\partial t}I_0\bigg(\frac{\lambda(1-p)}{c}\sqrt{c^2t^2-(u+v)^2}\bigg)\bigg]\nonumber\\
&\;\cdot\left[\lambda p\, I_0\left(\frac{\lambda p}{c}\sqrt{c^2t^2-(u-v)^2}\right)+\frac{\partial}{\partial t}I_0\left(\frac{\lambda p}{c}\sqrt{c^2t^2-(u-v)^2}\right)\right]\,du\,dv.\nonumber\end{align}
Therefore, in view of the transformation (\ref{tardi}), it follows that
\begin{align}\mathbb{P}\Big(&X(t)\in dx,\,Y(t)\in dy\Big)\nonumber\\=&\frac{e^{-\lambda t}}{2c^2xy}\left[\lambda(1-p) \, I_0\left(\frac{\lambda(1-p)}{c}\sqrt{c^2t^2-\log\left(\frac{xy}{x_0\,y_0}\right)^2}\,\right)\right.\nonumber\\
&\qquad\qquad\qquad+\left.\frac{\partial}{\partial t}I_0\left(\frac{\lambda(1-p)}{c}\sqrt{c^2t^2-\log\left(\frac{xy}{x_0\,y_0}\right)^2}\,\right)\right]\nonumber\\
&\qquad\cdot\left[\lambda p\, I_0\left(\frac{\lambda p}{c}\sqrt{c^2t^2-\log\left(\frac{x\,y_0}{x_0\,y}\right)^2}\,\right)\right.\nonumber\\
&\qquad\qquad\qquad+\left.\frac{\partial}{\partial t}I_0\left(\frac{\lambda p}{c}\sqrt{c^2t^2-\log\left(\frac{x\,y_0}{x_0\,y}\right)^2}\,\right)\right]\,dx\,dy.\label{fconclude}\end{align}
Observe that, in principle, the asymptotic distribution of the random vector $\big(X(t), Y(t)\big)$ can be obtained directly by taking the limit of the density function (\ref{fconclude}). Indeed, Orsingher \cite{orsingherlimiting} proved that
\begin{equation}\lim_{\lambda,c\to+\infty}\frac{e^{-\lambda t}}{2c}\left[\lambda\, I_0\left(\frac{\lambda}{c}\sqrt{c^2t^2-x^2}\right)\frac{\partial}{\partial t}I_0\left(\frac{\lambda}{c}\sqrt{c^2t^2-x^2}\,\right)\right]=\frac{1}{\sqrt{2\pi \sigma^2 t}}\,e^{-\frac{x^2}{2\sigma^2 t}}\label{limitinggauss}\end{equation} where the limit is taken under the scaling $\frac{\lambda}{c^2}\to\frac{1}{\sigma^2}$. By taking the limit of formula (\ref{fconclude}), and using (\ref{limitinggauss}), it follows that
\begin{align}\lim_{\lambda,c\to+\infty}&\mathbb{P}\Big(X(t)\in dx,\,Y(t)\in dy\Big)\nonumber\\
&=\frac{\sqrt{p(1-p)}}{\pi t\,x\,y}\,\exp\left\{-\frac{p\log\left(\frac{xy}{x_0y_0}\right)^2+(1-p)\log\left(\frac{x\,y_0}{x_0\,y}\right)^2}{2t}\right\},\; x,y>0\label{finalformularoby}\end{align} where the limit is taken in such a way that $\frac{\lambda}{c^2}\to1.$ Clearly, the density function (\ref{finalformularoby}) coincides with the distribution of the bivariate diffusion process (\ref{bivariateGBM}).  This confirms that the random vector $\big(X(t), Y(t)\big)$ converges in distribution, in the hydrodynamic limit, to a bivariate geometric Brownian motion with dependent components.

\bibliographystyle{plain}
\nocite{*}
\bibliography{bibliography}
\end{document}